\documentclass[12pt,reqno]{amsart}
\usepackage{lmodern}
\usepackage[english]{babel}
\usepackage[latin1]{inputenc}
\usepackage{microtype}
\usepackage{amsfonts,amssymb,amsmath,amsthm,mathrsfs}
\usepackage{mathtools}
\usepackage{latexsym} 
\usepackage[all]{xy}
\usepackage{pifont}  
\usepackage{enumerate}
\usepackage{dcolumn}
\usepackage{comment}
\usepackage[dvipsnames]{xcolor}
\newcolumntype{2}{D{.}{}{2.0}}
\xyoption{2cell}

\usepackage{hyperref}  
\newcommand{\cC}{\mathcal{C}}

\usepackage{tikz}
\usetikzlibrary{arrows}

\usepackage[inline]{enumitem}
\usepackage{cancel}
\usepackage[normalem]{ulem}
\usepackage{etoolbox}

\makeatletter
\patchcmd{\@setaddresses}{\indent}{\noindent}{}{}
\patchcmd{\@setaddresses}{\indent}{\noindent}{}{}
\patchcmd{\@setaddresses}{\indent}{\noindent}{}{}
\patchcmd{\@setaddresses}{\indent}{\noindent}{}{}
\makeatother

\makeatletter
\def\user@resume{resume}
\def\user@intermezzo{intermezzo}
\newcounter{previousequation}
\newcounter{lastsubequation}
\newcounter{savedparentequation}
\renewenvironment{subequations}[1][]{%
      \def\user@decides{#1}%
      \setcounter{previousequation}{\value{equation}}%
      \ifx\user@decides\user@resume 
           \setcounter{equation}{\value{savedparentequation}}%
      \else  
      \ifx\user@decides\user@intermezzo
           \refstepcounter{equation}%
      \else
           \setcounter{lastsubequation}{0}%
           \refstepcounter{equation}%
      \fi\fi
      \protected@edef\theHparentequation{%
          \@ifundefined {theHequation}\theequation \theHequation}%
      \protected@edef\theparentequation{\theequation}%
      \setcounter{parentequation}{\value{equation}}%
      \ifx\user@decides\user@resume 
           \setcounter{equation}{\value{lastsubequation}}%
         \else
           \setcounter{equation}{0}%
      \fi
      \def\theequation  {\theparentequation  \alph{equation}}%
      \def\theHequation {\theHparentequation \alph{equation}}%
      \ignorespaces
}{%
  \ifx\user@decides\user@resume
       \setcounter{lastsubequation}{\value{equation}}%
       \setcounter{equation}{\value{previousequation}}%
  \else
  \ifx\user@decides\user@intermezzo
       \setcounter{equation}{\value{parentequation}}%
  \else
       \setcounter{lastsubequation}{\value{equation}}%
       \setcounter{savedparentequation}{\value{parentequation}}%
       \setcounter{equation}{\value{parentequation}}%
  \fi\fi
  \ignorespacesafterend
}
\makeatother

  \def\tens{\mathop{\otimes}} 
  \def\<{{\langle}} 
  \def\>{{\rangle}}

  \def\eps{\varepsilon}

  \def\note#1{{}}

  \def\note#1{} 

  \def\cC{{\mathcal C}}

  \def\lhom#1#2#3{{}{{\rm Hom}\sb{#1}\left(#2,#3\right)}} 
  \def\rhom#1#2#3{{{\rm Hom}\sb{#1}\left(#2,#3\right)}}

\newcommand{\ring}{{\rm\bf Ring}} 

  \def\beq{\begin{equation}} 
  \def\eEq{\end{equation}}

\newcommand{\im}{{\rm Im}}

  \def\Hom{\mbox{\rm Hom}\,}

  \newcounter{blist} 
  \newenvironment{blist}{\begin{list}{(\alph{blist})}{ 
  \usecounter{blist}\leftmargin2.5em\labelwidth2em\labelsep0.5em 
  \topsep0.6ex 
  \parsep0.3ex plus0.2ex minus0.1ex}}{\end{list}} 
\def\stac#1{\raise-.2cm\hbox{$\stackrel{\displaystyle\otimes}{\scriptscriptstyle{#1}}$}}

\def\cten#1{\raise-.2cm\hbox{$\stackrel{\displaystyle\widehat{\otimes}}
{\scriptscriptstyle{#1}}$}}

  \def\Label#1{\label{#1}\ifmmode\llap{[#1] }\else 
  \marginpar{\smash{\hbox{\tiny [#1]}}}\fi} 
  \def\Label{\label}

  \newtheorem{proposition}{Proposition}[section]
  \newtheorem{lemma}[proposition]{Lemma} 
  \newtheorem{corollary}[proposition]{Corollary} 
  \newtheorem{theorem}[proposition]{Theorem} 

\theoremstyle{definition} 
  \newtheorem{definition}[proposition]{Definition}
    
  \newtheorem{example}[proposition]{Example}

  \theoremstyle{remark} 
  \newtheorem{remark}[proposition]{Remark}

   \numberwithin{equation}{section}

\newcommand{\Mod}{{\mathbf{mod}}}

\newcommand{\ot}{\otimes}

\newcommand{\NN}{{\mathbb N}}

\newcommand{\TT}{{\mathbb T}}

\newcommand{\ZZ}{{\mathbb Z}}

\newcommand{\gG}{\mathrm{G}}
\newcommand{\hH}{\mathrm{H}}

\newcommand{\tT}{\mathrm{T}}

\newcommand{\Aa}{\mathcal{A}}

\newcommand{\Ff}{\mathcal{F}}

\newcommand{\Hh}{\mathcal{H}}

\newcommand{\Tt}{\mathcal{T}}
\newcommand{\Uu}{\mathcal{U}}

\newcommand{\Set}{\mathbf{Set}}
\newcommand{\ahrd}{\mathbf{Ah}}

\newcommand{\trs}{\mathbf{Trs}}
\newcommand{\lto}{\longmapsto}
\newcommand{\lra}{\longrightarrow}
\newcommand{\sym}[1]{\colon\!\!#1\!\colon\!\!}

\newcommand{\ds}{\hbox{-}}

\newcommand{\boks}[3]{\overset{#2}{\underset{#1}{\boxplus}}\,#3}
\newcommand{\Abs}{\mathrm{Abs}}

\newcommand{\pMod}[1]{{#1}\ds\Mod}

\newcommand{\lmod}[1]{{#1}\ds\Mod}
\newcommand{\rmod}[1]{\Mod\ds{#1}}

\newcommand{\ie}{that is,~}
\newcommand{\LL}{\left[}
\newcommand{\RR}{\right]}
\newcommand{\ev}{\mathsf{ev}}
\newcommand{\db}{\mathsf{db}}
\newcommand{\id}{\mathsf{id}}
\newcommand{\op}[1]{{#1}^\circ}
\newcommand{\bs}[1]{\boldsymbol{#1}}

\newcommand{\Ker}{\underline{\mathrm{Ker}}}

\newcommand{\uT}{T_u} 
\newcommand{\uS}{S_u} 

\title{On functors between categories of modules over trusses }

\author{Tomasz Brzezi\'nski}

\address{
Department of Mathematics, Swansea University, 
Swansea University Bay Campus,
Fabian Way,
Swansea,
 Swansea SA1 8EN, U.K.\ \newline
Faculty of Mathematics, University of Bia{\l}ystok, K.\ Cio{\l}kowskiego  1M,
15-245 Bia\-{\l}ys\-tok, Poland}

\email{T.Brzezinski@swansea.ac.uk}

\author{Bernard  Rybo{\l}owicz}

\address{
Department of Mathematics, 
Heriot-Watt University, 
Edinburgh EH14 4AS, U.K.}

\urladdr{\url{https://sites.google.com/view/bernardrybolowicz/}}
\email{B.Rybolowicz@hw.ac.uk}

\author{Paolo Saracco}

\address{
D\'epartement de Math\'ematique, Universit\'e Libre de Bruxelles, 
Bd du Triomphe, B-1050 Brussels, Belgium.}

\urladdr{sites.google.com/view/paolo-saracco}
\urladdr{paolo.saracco.web.ulb.be}

\email{paolo.saracco@ulb.be}

\thanks{This version of the article has been accepted for publication, after peer review. The final publication is available at Elsevier via \href{https://doi.org/10.1016/j.jpaa.2022.107091}{doi.org/10.1016/j.jpaa.2022.107091}.}



\usepackage{fancyhdr}

\begin{document}

\begin{abstract}
Categorical aspects of the theory of modules over trusses are studied. Tensor product of modules over trusses is defined and its existence established. In particular, it is shown that  bimodules over  trusses form a monoidal category. Truss versions of the Eilenberg-Watts theorem and Morita equivalence are formulated.  Projective and small-projective modules over trusses are defined and their properties studied.
\end{abstract}
\date\today
\maketitle

\fancyhf{}
\renewcommand{\headrulewidth}{0pt}
\thispagestyle{fancy}
\cfoot{\smallskip\footnotesize $\begin{gathered}\includegraphics[scale=0.5]{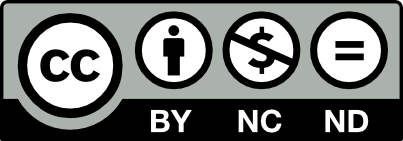}\end{gathered}$ \ \copyright\, 2022. This manuscript version is made available under the \href{https://creativecommons.org/licenses/by-nc-nd/4.0/}{CC-BY-NC-ND 4.0} license.}

\tableofcontents

\section{Introduction}\label{sec.intro}

Introduced in the 1920s by Pr\"ufer \cite{Pru:the} and Baer \cite{Bae:ein}, heaps can be understood as affinizations of groups, that is, as algebraic structures composed by a set $H$ and a ternary operation $[-,-,-]$ from which, upon specifying an element, a group structure on $H$ can be obtained with this element as identity.  In a similar vein, trusses can be understood as affine versions of rings in which the associative multiplication distributes over a ternary abelian heap operation rather than over a binary abelian group one. Most interesting, and motivationally important for their study, is the observation that trusses can also be understood as affine versions of (two-sided) braces, originally introduced by Rump \cite{Rum:bra} in the context of the set-theoretic Yang-Baxter equation.

In their primitive form, (skew) trusses appeared in \cite{Brz:tru} in order to grasp the nature of the law binding two operations together into a skew brace. However, it has been realized soon that trusses and their more general analogues (such as pre-trusses, near-trusses and skew trusses) were carrying an intrinsic interest on their own and not only for the interactions with ring theory and braces. For a glimpse of the increasing attention they are attracting, we refer the reader to the recent publications \cite{BrzMerRyb:pre, BrzRyb:cong, ColAnt}. In particular, it deserves to be mentioned how the greater flexibility offered by trusses with respect to rings allowed for the extension of the celebrated Baer-Kaplansky Theorem \cite[Theorem 16.2.5]{Fu15} to all abelian groups, provided that the endomorphism rings are replaced by the endomorphism trusses (see \cite{BrBr}).

A systematic study of trusses and their modules, mainly from a ring-theoretical point of view, has been initiated in \cite{Brz:par} and carried on in \cite{BrzMerRyb:pre, BrzRyb:cong, BrzRyb:mod}, and it appears to be leading to a successful framework in which both ring and brace theoretical modules can be treated on equal footing as different specifications of the same algebraic system.  We believe that this unification, which allows one to transfer concepts from one to the other theory, largely rewards the additional effort required by handling the slightly less familiar notion of a truss. 

Nevertheless, and despite a seeming similarity between trusses and rings, there are substantial differences in the behaviour of modules over trusses compared with modules over rings, which break the apparent closeness of the two theories. The origin of these differences can be traced back to the facts that \begin{enumerate*} \item neither the algebraic structure of a truss nor -- in a more pronounced way (and applicable even to the case of unital modules over trusses with identity) -- of a module over a truss includes a nullary operation, and
 \item the transition from a truss to a ring structure by fixing a neutral element is obviously not functorial (as it depends on an arbitrary choice of an element)\end{enumerate*}. Since no choice of a distinguished element needs  -- or is advised -- to be made (and so preserved by homomorphisms), one reasonably expects the category of modules over a truss being richer than that of modules over a ring. For instance, empty objects can be considered too and hom-sets would not be abelian groups in general, leading immediately to the observation that the category of modules over a truss, although sharing many properties with a category of modules over a ring, cannot be abelian and cannot admit any natural forgetful functor to an abelian category. The divergence between modules over rings and modules over trusses becomes even more evident in the separation of finite products from finite coproducts (initially observed in \cite{BrzRyb:mod}), which has profound consequences for the characterisations of projective objects. 

Supported by the aforementioned considerations, the aim of this work is to continue a systematic study of modules over trusses, now with special attention devoted to their categorical properties. 
Our approach is based on analysing, comparing and contrasting different categorical properties, which are well-known for modules over rings, but completely new for modules over trusses. On the one hand, this sets the stage for a unified approach to ring and brace theoretical modules, as suggested (and hoped for) in the initial paragraphs.
On the other hand, bringing to light and inspecting the monoidal structure of the categories of abelian heaps and of modules over trusses opens the way for a better understanding of categories enriched over abelian heaps rather than over abelian groups (or vector spaces), for instance, categories enriched over affine spaces.

We start this paper with preliminary Section~\ref{sec.pre} in which we review basic properties of heaps, trusses and modules over trusses, including the construction of direct sums and an initial study of relations between modules over a ring and modules over the associated truss firstly presented in \cite{BrzRyb:mod}. The results in Section~\ref{ssec.unital}, where we study some categorical consequences of the unital extensions for trusses, and Section~\ref{sec.cat}, where we connect epimorphisms of abelian heaps (or modules over a truss) with coequalizers (which we describe explicitly) and monomorphisms with equalizers, can be considered as possessing a moderate level of novelty.  Next we focus on the following topics. We define the tensor product of modules over trusses by its universal property and we present an explicit construction that affirms the existence of the tensor product for all modules in Theorem~\ref{thm.tensor}, the first main result of the paper. Then we show in Proposition~\ref{prop.tens.act} that, just as in the case of modules over rings, tensoring with a bimodule over trusses defines a functor between the corresponding categories of modules. This functor enjoys properties familiar from ring theory: for example, it is the left adjoint functor to the hom-functor. Thus, the tensor product is an associative operation that defines a monoidal structure on the category of bimodules over any truss, in particular on the category of heaps.  Consequently, the axioms of trusses and their modules can be formulated internally within the latter: a truss is a semigroup in the category of heaps, a unital truss is a monoid in this category, modules over trusses are modules over these internal structures. Finally, again in perfect accord with ring theory, the tensor product over trusses is a coequalizer of morphisms induced by actions.

Once the properties of the tensor product are established, we turn our attention to functors between categories of modules over trusses. In particular, the second main result of this paper, Theorem~\ref{thm:EW}, is an analogue for modules over trusses of the celebrated Eilenberg-Watts Theorem: it states that cocontinuous heap functors between categories of (unital) modules are necessarily given by tensoring with a (unital) bimodule. In the third main result, Theorem~\ref{thm:Morita}, we derive a Morita-type characterisation of equivalences between categories of  modules over trusses. Similarly to the case of rings, this characterisation is provided by strict Morita contexts consisting of two bimodules and dual basis and evaluation isomorphisms.

As desired (and somehow expected), the so far presented properties of modules over trusses quite faithfully mirror those of modules over rings. The paths start to diverge once a closer look is taken at those bimodules satisfying the hypotheses of Theorem \ref{thm:Morita} and, more generally, at preservation properties of the hom-functors from the category of modules over a truss $T$ to the category of abelian heaps, that is, at the projectivity of $T$-modules. Being a right adjoint functor, $\lhom T P -$ is continuous, that is, it preserves limits and hence, in particular, monomorphisms. However, as in the classical framework, it is not necessarily cocontinuous nor even right exact, whence it does not preserve colimits or epimorphisms in general. Keeping Theorem \ref{thm:Morita} in mind, we say that a finitely generated  $T$-module $P$ for which $\lhom T P -$ is right exact is a {\em tiny module} or a {\em small-projective module}. Such a module is equivalently characterised by the existence of a finite dual basis or by the identification of $\lhom T P -$ with the tensor functor ${^*P}\ot_T-$, where ${^*P}$ is the dual module; see Theorem~\ref{thm:toostrong}.  These modules reflect (small) projectivity over rings in the following sense. A finitely generated projective module over a ring is tiny over the associated truss; the absorber functor, that sends modules over the truss associated to a ring to modules over this ring, applied to a tiny module  yields  a finitely generated projective module; see Proposition~\ref{prop:Absproj}. Alas, despite the new insights into ring and module theory offered by the novel paradigm, the extent to which genuine tiny modules over trusses exist (as opposed to modules over rings seen as modules over trusses) is not clear at present.
This leads one to consider $T$-modules $P$  satisfying the (to be shown, weaker) condition that $\lhom T P -$ preserves epimorphisms and which we call {\em projective}, in accordance with the classical theory.  

Since the category of modules over a ring is an abelian category, the preservation of epimorphisms is sufficient for the exactness of the hom-functor. In contrast, as we already mentioned, the category of modules over a truss is enriched over the category of abelian heaps and not over the category of abelian groups, hence it is not (and should not be expected to be) an abelian or even pre-additive category in the usual sense. Therefore, one reasonably expects that the properties of being tiny or finitely generated and projective for modules over trusses do not coincide any longer (see Example \ref{ex:tinyfree}). 
In spite of the fact that the  category of modules over a truss is not abelian, some ways of defining an exact sequence are  still possible and we explore one such possibility in Section~\ref{sub.ses}. In contrast with the category of abelian groups or modules over a ring, the splitting of an exact sequence gains here two different meanings. The splitting of the monomorphism (that is, the existence of a retraction) allows one to identify the  middle module as the product of the outer modules; see Proposition~\ref{prop:prod}. On the other hand, the splitting of the epimorphism (that is, the existence of a section) identifies the  middle module as the product of the third module with the first module but with an {\em induced} structure; see Proposition~\ref{prop:starsum}. In fact, the new module structure realizes the first heap as the quotient module of the second by the third. This is the first place where the fact that finite products of abelian heaps (or modules over trusses) differ from finite coproducts truly shows up. This further prevents one from deducing that the preservation of epimorphisms yields exactness of the hom-functor, as mentioned above, and also affects the possibility of characterising projective modules in several ways. Although we develop  dual bases  and show that a module with a dual basis is projective, it is not clear whether the projectivity guarantees the existence of a dual basis.  In fact, we conclude the paper by showing that a non-empty module over a  truss is projective if and only if it is a direct \emph{factor} (not a direct \emph{summand}!) of a free module, with the other factor possessing an absorber; see Theorem~\ref{thm.proj}.

\section{Heaps, trusses, modules}\label{sec.pre}

\subsection{Heaps} \label{sec.heap}
A {\em heap} (\cite{Bae:ein}, \cite{Pru:the} or \cite{Sus:the})  is a  set $H$ together with a ternary operation $
[---]: H\times H\times H\lra H$ which is associative and satisfies the Mal'cev identities, \ie
\begin{equation}\label{heap.def}
[[a,b,c],d,e] = [a,b,[c,d,e]] \qquad \text{and} \qquad [a,b,b] = a = [b,b,a]
\end{equation}
 for all $a,b,c,d,e\in H$.  
 A morphism of heaps  is a function that preserves ternary operations. 
  A singleton set with the (unique) ternary operation is the terminal object in 
 the category of heaps; we denote it by $\star$. The empty set with the empty ternary operation is the initial object in this category;  we denote it by $\varnothing$.

A heap $H$ is said to be {\em abelian} if for all $a,b,c\in H$,
\begin{equation}\label{ab.heap}
[a,b,c] = [c,b,a].
\end{equation}
The full subcategory of  the category of heaps consisting of abelian heaps is denoted by $\ahrd$. Homomorphism sets of abelian heaps are themselves abelian heaps with the point-wise operation. Clearly, both $\star$ and $\varnothing$ are abelian heaps and they are the terminal and initial object, respectively, also in $\ahrd$.

There is a close relationship between heaps and groups. Given a (abelian) group $G$, there is an associated (abelian) heap $\hH(G)$ with operation, for all $g,h,k\in G$,
\begin{equation}\label{heap-form-gr}
[g,h,k] = gh^{-1}k.
\end{equation}
This  assignment of the heap $\hH(G)$ to the group $G$ is a functor from the category of (abelian) groups to that of (abelian) heaps. Conversely, a choice of any element $e$ in a non-empty heap $H$ defines the group $\gG(H;e)$, known as  the {\em $e$-retract of $H$}, with neutral element $e$ and composition law 
\begin{equation}\label{gr-from-heap}
a\cdot_e b = [a,e,b], \qquad \mbox{for all $a,b\in H$}. 
\end{equation}
The inverse of $a\in \gG(H;e)$ is $a^{-1}= [e,a,e]$. 
For all $e,f,\in H$, the functions
\begin{equation}\label{swap}
\tau_e^f : \gG(H;e)\lra \gG(H;f), \qquad a\lto [a,e,f], 
\end{equation}
are  group isomorphisms. The process of converting heaps into groups and groups into heaps is asymmetric on two levels. First, since the conversion of a heap into a group involves a choice of an element, this is not functorial, while the opposite operation is given by a functor. The second key difference is  best expressed by the following formulae
\begin{equation}\label{h-g-h-g}
\hH\Big(\gG\big(H;e\big),\cdot_e,e\Big) = H, \qquad \Big(\gG\big(\hH\left(G\right);f\big),\cdot_f,f\Big) \cong \Big(\gG\big(\hH\left(G\right);e\big),\cdot_e,e\Big)
\end{equation}
for all heaps $H$ and groups $G$. For the sake of clarity, notice that if $e \in G$ is the neutral element of $G$, then $\Big(\gG\big(\hH\left(G\right);e\big),\cdot_e,e\Big) = G$ and the isomorphism $\gG\big(\hH\left(G\right);f\big) \cong G$ is explicitly given by
\[
\begin{gathered}
\xymatrix @R=0pt {
G \ar[r] & \gG\big(\hH\left(G\right);f\big) &  
g \ar@{|->}[r] & gf\ .
}
\end{gathered}
\]
Equations \eqref{heap.def} imply that, for all $a,b,c,d,e\in H$,
\begin{equation}\label{def.heap.con}
[[a,b,c],d,e]=[a,[d,c,b],e] = [a,b,[c,d,e]].
\end{equation}
Consequently, in the case of an abelian heap, the reduction obtained by any placement of brackets in a sequence of elements of $H$ of odd length yields the same result. In this case we write
\begin{equation}\label{multi}
[a_1,\ldots, a_{2n+1}]_n \quad \mbox{or simply}\quad [a_1,\ldots, a_{2n+1}], \qquad a_1,\ldots , a_{2n+1}\in H,
\end{equation}
for the result of applying the (abelian) heap operation $n$-times in any possible way. Furthermore, in an abelian heap one can resort to the following {\em transposition rule} (see \cite[Lemma 2.3]{Brz:par}),
\begin{equation}\label{tran}
[[a_1,a_2,a_3],[b_1,b_2,b_3],[c_1,c_2,c_3]] = [[a_1,b_1,c_1],[a_2,b_2,c_2],[a_3,b_3,c_3]].
\end{equation}

The following lemma formalizes the extension of the transposition rule to an arbitrary family of elements.

\begin{lemma} 
In an abelian heap $H$,
\begin{equation}\label{eq:superbracket}
\LL\LL a_{i,j} \RR_{i=1}^{2n+1} \RR_{j=1}^{2m+1} = \LL\LL a_{i,j} \RR_{j=1}^{2m+1} \RR_{i=1}^{2n+1}.
\end{equation}
for all $n,m\geq 0$ and for all $a_{i,j}\in H$, $i=1,\ldots,2n+1,j=1,\ldots,2m+1$.
\end{lemma}
\begin{proof}
One can show by induction that 
\begin{equation}\label{eq:qsuperbracket}
\LL a_i,b_i,c_i\RR_{i=1}^{2n+1} = \left[\LL a_i\RR_{i=1}^{2n+1}, \LL b_i\RR_{i=1}^{2n+1}, \LL c_i\RR_{i=1}^{2n+1}\right],
\end{equation}
for all $n\geq 0$, and then also that 

\[
\LL\LL a_{i,j} \RR_{i=1}^{2n+1} \RR_{j=1}^{2k+1} = \LL\LL a_{i,j} \RR_{j=1}^{2k+1} \RR_{i=1}^{2n+1},
\]
 for all $k\geq 0$, in a similar way.
\end{proof}

Directly from equations \eqref{heap.def} one can also observe that adding or removing an element in two consecutive places, whether separated by a bracket or not, does not change the value of the (multiple) heap operation. Another important consequence of the definition of a heap is that if for $a,b\in H$ there exists $c\in H$ such that 
\begin{equation}\label{equal}
[a,b,c] = c \quad \mbox{or} \quad [c,a,b]=c,
\end{equation}
then $a=b$. In fact, in view of the Mal'cev identities, \eqref{equal} is an equivalent characterisation of equality of elements in a heap.

\subsection{Sub-heaps and the sub-heap equivalence relation.}
A subset $S$ of a heap $H$ that is closed under the heap operation is called a {\em sub-heap} of $H$. 
Every non-empty sub-heap $S$ of an abelian heap $H$ defines a congruence relation $\sim_S$ on $H$:
\begin{equation}\label{rel.sub}
a\sim_S b \quad \iff \quad \exists s\in S,\ [a,b,s]\in S \quad \iff \quad \forall s\in S,\ [a,b,s]\in S.
\end{equation}
 The equivalence classes of $\sim_S$ form an abelian heap with operation induced from that in $H$. Namely,
$
[\bar a, \bar b, \bar c] = \overline{[a,b,c]}
$, where $\bar x$ denotes the class of $x$ in $H/\sim_S$ for all $x\in H$. This is known as the {\em quotient heap} and it is denoted by $H/S$. 
For any $s\in S$ the class of $s$ is equal to $S$.

If $\varphi: H\lra K$ is a morphism of abelian heaps, then for all $e\in \im (\varphi)$ the set
\begin{equation}\label{e-ker}
\ker_e (\varphi) \coloneqq  \{ a\in K\; |\; \varphi(a)=e\}
\end{equation}
is a sub-heap of $K$. Different choices of $e$ yield isomorphic sub-heaps and the quotient heap $H/\ker_e (\varphi)$ does not depend on the choice of $e$. Moreover, the sub-heap relation $\sim_{\ker_e (\varphi)}$ is the same as the {\em kernel relation} defined by: $a\,\Ker(\varphi)\,b$ if and only if $\varphi(a) = \varphi(b)$ (see \cite[Lemma 2.12(3)]{Brz:par}). Thus we write $\ker(\varphi)$ for $\ker_e (\varphi)$ and we refer to it as the {\em kernel} of $\varphi$.

The following fact, concerning kernels and heap homomorphisms, has been used more or less implicitly a number of times in the development of trusses and their modules. Therefore, we believe it would be useful to state and prove it explicitly at least once.

\begin{lemma}\label{lem:quomap}
Let $\varphi: A \lra B$ be a morphism of abelian heaps and $S\subseteq A$ be a sub-heap.  Denote  by $\pi: A\lra A/S$, $a \lto \bar{a},$ the canonical projection. 
 If the sub-heap relation $\sim_S$ is a sub-relation of the kernel relation $\Ker(\varphi)$, 
then there exists a unique morphism of abelian heaps $\tilde{\varphi} : A/S \lra B$ rendering the following diagram
\[
\xymatrix @=20pt{
A \ar[dr]_-{\varphi} \ar[rr]^-{\pi} & & \displaystyle A/S \ar@{.>}[dl]^-{\tilde{\varphi}} \\ 
& B & 
}
\]
commutative. In particular, if $S\subseteq \ker_e(\varphi)$ for a certain $e \in B$, then the conclusion follows.
\end{lemma}

\begin{proof}
If $A$ is the empty heap, then there is nothing to prove. Thus, assume that $A$ is not the empty heap, which implies that $B$ is non-empty as well. 

Uniqueness of $\tilde{\varphi}$ follows from the surjectivity of $\pi$. Therefore, let us check that
\[
\tilde{\varphi}: A/S \lra B, \qquad \bar{a} \lto \varphi(a),
\]
is a well-defined heap homomorphism. If $\bar{a} = \bar{b}$, then $a \sim_S b$ and so $a\,\Ker(\varphi)\, b$ as well, hence $\varphi(a) = \varphi(b)$. Thus, $\tilde{\varphi}$ is independent of the choice of the representative. 

Furthermore, if there exists $e \in B$ such that $S\subseteq \ker_e(\varphi)$, then $\sim_S$ is a sub-relation of $\Ker(\varphi)$, since $\Ker(\varphi) = ~ \sim_{\ker_e(\varphi)}$. Explicitly, since $[a,b,s] = s' \in S$,
\[
e = \varphi(s') = \varphi([a,b,s]) = [\varphi(a),\varphi(b),\varphi(s)] = [\varphi(a),\varphi(b),e],
\]
which entails that $\varphi(a) = \varphi(b)$.
\end{proof}

In the case of $B=\im(\varphi)$ and $S=\Ker(\varphi)$, the induced map $\tilde{\varphi}$ is an isomorphism that establishes the standard first isomorphism theorem for heaps: $\im(\varphi)\cong A/\Ker(\varphi)$.

For any non-empty subset $X$ of a heap $H$, the sub-heap generated by $X$, denoted by $\langle X\rangle$, is equal to the intersection of all sub-heaps containing $X$. If $X$ is a singleton set, then $ \langle X\rangle = X$. If $H$ is an abelian heap then $\langle X\rangle$ can be described explicitly as
\begin{equation}\label{gen.sub}
\langle X\rangle = \{[x_1,\ldots, x_{2n+1}]\; |\; n\in \NN, x_i \in X\}.
\end{equation}

\subsection{Trusses and their modules} \label{sec.truss}
Recall from \cite{Brz:tru} or \cite{Brz:par} that a {\em truss} is an abelian heap $T$ together with an associative binary operation (denoted by juxtaposition and called {\em multiplication}) that distributes over the heap operation, \ie for all $s,t,t',t''\in T$,
\begin{equation}\label{truss.dist}
s[t,t',t''] = [st, st', st''] \quad \mbox{and} \quad [t,t',t'']s = [ts, t's, t''s].
\end{equation}
A truss is said to be {\em unital} (or \emph{to have identity}) if there is a (necessarily unique) neutral element for its multiplication. The identity is denoted by 1. If $T$ is a truss then $T$ with opposite multiplication is a truss too, called the {\em opposite} truss and denoted by $T^\circ$. A fundamental example of a (unital) truss is the {\em endomorphism truss} of an abelian heap, $E(H) = \ahrd(H,H)$, which has the pointwise defined heap operation 
and multiplication given by the composition of morphisms. Equivalently, $E(H)$ can be seen as a semi-direct product of any (isomorphic) $e$-retract  $\gG(H;e)$ of $H$ (see \eqref{gr-from-heap}) with the endomorphism monoid of  $\gG(H;e)$ (see \cite[Proposition 3.44]{Brz:par}).

A heap homomorphism between two trusses is a truss homomorphism if it respects multiplications. The category of trusses and their morphisms is denoted by $\trs$. In case of unital trusses we require in addition that morphisms preserve identities; the corresponding category is denoted by $\trs_1$. In an obvious way, the terminal object $\star$ (\ie\ the singleton set with the unique ternary operation) of the category $\ahrd$ is also the terminal object of the category $\trs$ of trusses and the zero object (both initial and terminal) of the category $\trs_1$ of unital ones.

\begin{remark}
Even if a unital truss may be informally described as a unital ring whose underlying group structure has no specified neutral element, the last observation of the previous paragraph should already convince the reader that truss theory is intrinsically different from ring theory. The category $\ring$ of unital rings admits an initial object, the ring of integers $\ZZ$, and a terminal one, the zero ring $\bs{0}$. It does not admit the zero object though (it is self-evident that $\ZZ\neq \bs{0}$).
\end{remark}

Let $T$ be a truss. A {\em left $T$-module} is an abelian heap $M$ together with an associative left action $\lambda_M: T\times M \lra M$ of $T$ on $M$ that distributes over the heap operation. The action is denoted on elements by $t\cdot m = \lambda_M(t,m)$, with $t\in T$ and $m\in M$. Explicitly, the axioms of an action state that, for all $t,t',t''\in T$ and $m,m',m''\in M$,
\begin{subequations}\label{module}
\begin{equation}\label{module1}
t\cdot(t'\cdot m) = (tt')\cdot m, 
\end{equation}
\begin{equation}\label{module2}
 [t,t',t'']\cdot m = [t\cdot m,t'\cdot m,t''\cdot m] , 
\end{equation}
\begin{equation}\label{module3}
 t\cdot [m,m',m''] = [t\cdot m,t\cdot m',t\cdot m''].
\end{equation}
\end{subequations}
If $T$ is a unital  truss and the action satisfies $1\cdot m = m$, then we say that $M$ is a {\em unital} or {\em normalised} module. Equivalently, a (unital) $T$-module can be described as an abelian heap $M$ together with a homomorphism of (unital) trusses $T \lra E(M)$.

A module homomorphism is a homomorphism of heaps between two modules that also respects the actions. As it is customary in ring theory we often refer to homomorphisms of $T$-modules as to $T$-linear maps or morphisms.
The category of left $T$-modules is denoted by $T$-$\Mod,$ that of left unital $T$-modules  by $ T_{1}\ds\Mod$ and the heaps of homomorphisms between modules $M$ and $N$ are denoted by $\lhom T MN$.  A left module for the truss $T^\circ$  opposite to $T$ is  called a {\em right} $T$-module. The category of right $T$-modules is denoted by $\Mod$-$T,$ that of right unital $T$-modules  by $\Mod\ds T_{1}$ and abelian heaps of morphisms are denoted by $\rhom T MN$ (the side being clear from the context). The category of $T\ds T'$-bimodules will be denoted by $T\ds\Mod \ds T'$ and analogously the category of unital $T\ds T'$-bimodules by $T_{1}\ds\Mod \ds T'_{1}.$   The terminal heap $\star$ and the initial heap $\varnothing$, with the unique possible actions, are the terminal and the initial object, respectively, in $T$-$\Mod$ and $\Mod$-$T$. It is remarkable that, since $\star \neq \varnothing$, $\lmod{T}$ and $\rmod{T}$ do not have zero object.

The category of left (right or two-sided) $T$-modules is enriched over the category $\left(\ahrd,\times,\star\right)$ of abelian heaps. In particular, $\rhom T M M \subseteq \ahrd (M,M)$ is a unital sub-truss of the unital endomorphism truss $E(M)$. We denote it by $E_T(M)$. 
We say that a functor $F: T\ds\Mod \lra T'\ds\Mod$ is a {\em heap functor} if it induces a heap homomorphism between the hom-sets, that is, if for all $M,N\in T\ds\Mod$, the induced function
\begin{equation}\label{heap.fun}
F_{M,N} : \lhom T MN\lra \lhom{T'} {F(M)}{F(N)},\qquad \varphi \lto F(\varphi),
\end{equation}
is a homomorphism of abelian heaps. Since functors preserve compositions and identities, $F_{M,M}: E_T(M)\lra E_{T'}(F(M))$ is a morphism of unital trusses for all $M\in \lmod{T}$.

An element $e$ of a left $T$-module $M$ is called an {\em absorber} provided that
\begin{equation}\label{absorb}
t\cdot e = e, \qquad \mbox{for all $t\in T$}.
\end{equation}
If $e$ is an absorber, then the action (left) distributes over the abelian group operation on $M$ associated to $e$ as in \eqref{gr-from-heap}. The set of all absorbers in $M$  is denoted by $\Abs(M) = \left\{ m\in M \mid r\cdot m = m, \ \forall\,r\in R\right\}$. 

Given a left (respectively, right) $T$-module $M$ and an element $e\in M$, the {\em $e$-induced action} of $T$ on $M$ is defined for all $t\in T, m\in M$, by
\begin{equation}\label{induced}
t\cdot_e m \coloneqq  [t\cdot m, t\cdot e, e], \qquad\mbox{(respectively, $m\cdot_e t \coloneqq  [m\cdot t, e\cdot t, e]$)}.
\end{equation}
$M$ is a left (respectively, right) $T$-module with action \eqref{induced} in which $e$ is an absorber. We refer to $M$ with this action as to an {\em induced} module and denote it by $M^{(e)}$. Different choices of $e$ yield isomorphic induced modules and an iteration of an induced action gives an induced action (see \cite[Lemma 4.29]{Brz:par}).

A sub-heap $N$ of a left $T$-module $M$ is called a {\em submodule} if it is closed under the $T$-action. A sub-heap $N$ is called an {\em induced submodule} if there exists $e\in N$ such that $N$ is a submodule of the  induced module $M^{(e)}$. This implies (in fact is equivalent to) that $N$ is a submodule of any $M^{(n)}$, $n\in N$.  If $N$ is a sub-heap of a left $T$-module $M$, then $M/N$ is a $T$-module such that the canonical epimorphism $M\lra M/N$ is a $T$-module homomorphism if and only if $N$ is an induced submodule of $M$ (see \cite[Proposition 4.32(2)]{Brz:par}). A kernel of a module homomorphism as defined by \eqref{e-ker} is an induced submodule of the domain. Furthermore, any $T$-module homomorphism $\varphi: M\lra M'$ factorizes uniquely as a $T$-linear map through the canonical epimorphism $M\lra M/N$ for any induced submodule $N$ of $M$ contained in $\ker_e(\varphi)$ as in Lemma~\ref{lem:quomap}. In particular, for all $T$-module morphisms with domain $M$, this yields an analogue of the first isomorphism theorem for $T$-modules: $\im(\varphi) \cong M/\Ker(\varphi)$.

If $N$ is a submodule of $M$, then it is automatically an induced submodule (with respect to any of its elements). In that case $N$ is an absorber of the corresponding quotient $T$-module $M/N$ and all quotient modules with absorbers can be identified as those arising as a quotient by submodules. The image of a module homomorphism is a submodule of the codomain.

Any abelian heap $H$ is a module over the terminal truss $\star$ with the action given by any endomorphism of $H$. Obviously, it is a unital module over $\star$ in a unique way with the identity action. Thus abelian heaps can be identified with unital modules over $\star$.
A morphism of trusses $\varphi: T\lra S$ induces a change of scalars functor 
$S\ds\Mod \lra T\ds\Mod$: an $S$-module $M$ is a $T$-module with action $t\cdot m = \varphi(t)\cdot m$. In particular, any abelian heap $H$ is a module over any truss $T$ through the action of $\star$ on $H$ and the unique morphism $T\lra \star$. 

\subsection{Products and direct sums of modules.}\label{ssec.(co)prod}
Given left $T$-modules $M$ and $N$ their product $M\times N$ has the left $T$-module structure defined component-wise, that is,
$$
\left[(m,n),(m',n'),(m'',n'')\right] =  \left([m,m',m''],[n,n',n'']\right), \quad t\cdot (m,n) = (t\cdot m,t\cdot n),
$$
for all $t\in T$, $m,m',m''\in M$ and $n,n',n''\in N$.

The definition of the coproduct  or {\em direct sum} of  $T$-modules has been presented in \cite[\S3]{BrzRyb:mod} and it is based on the coproduct of abelian heaps discussed exhaustively therein. We summarise this definition presently.

Let $X$ be a set. The free abelian heap on $X$, denoted $\Aa(X)$, is constructed in two steps. First one constructs the free heap $\Hh(X)$ by considering all words of odd length from the alphabet $X$ in which no two consecutive letters are the same. These are referred to as {\em reduced words}. The heap operation on words $w_1$, $w_2$, $w_3$ is obtained by concatenation of $w_1w_2^\circ w_3$, where $w_2^\circ$ means $w_2$ written backwards, and removal of all pairs of identical neighbours. In this way a reduced word of odd length is obtained again and one can check that this procedure defines a heap operation and completes the construction of the free heap $\Hh(X)$.  The second step involves the symmetrisation of the words in $\Hh(X)$. A {\em symmetric word} of odd length in the alphabet $X$ is defined, for all $x_1,\ldots x_{n+1}, y_1,\ldots, y_n \in X$,  as the set
$$
w= \sym{x_1y_1x_2\ldots y_nx_{n+1}} =\{ x_{\sigma(1)}y_{\hat\sigma(1)}x_{\sigma(2)}\ldots y_{\hat\sigma(n)}x_{\sigma(n+1)}\; |\; \sigma\in S_{n+1}, \hat\sigma\in S_n\},
$$
where $S_k$ denotes the symmetric group. The word $w$ is said to be {\em reduced} if all the words included in $w$ are reduced. One can easily check that all symmetric words are equivalence classes of a congruence relation on $\Hh(X)$ and thus they form a heap $\Aa(X)$  which is abelian, by the symmetrisation. Due to the definition of the heap structure on $\Aa(X)$, we will often denote a symmetric word $w= \sym{x_1y_1x_2\ldots y_nx_{n+1}}$ by
\[
\left[x_1,y_1,x_2,\ldots, y_n,x_{n+1}\right].
\]

The \emph{direct sum} or \emph{coproduct} of abelian heaps $M$ and $N$, denoted by $M\boxplus N$, is the quotient of the free abelian heap $\Aa(M\sqcup N)$ on the disjoint union $M\sqcup N$ by the sub-heap generated by  
$$
[[m,m',m''],[m,m',m'']_{M},e],\quad [[n,n',n''],[n,n',n'']_{N},e],
$$
where $m,m',m''\in M,$ $ n,n',n''\in N$, $[---],[---]_{M},[---]_{N}$ are the ternary operations in $\Aa(M\sqcup N)$, $M$ and $N$, respectively, and $e$ is any fixed element of $\Aa(M\sqcup N)$. In other words, to obtain $M\boxplus N$ we consider all words in $\Aa(M\sqcup N)$ and apply ternary operations in $M$ (respectively $N$) wherever there are three consecutive elements in $M$ (respectively $N$) in any representative of symmetric words. It has been shown in \cite[Proposition 3.6]{BrzRyb:mod} that any representative can be reduced to one of the six forms $m$, $n$, $mnn'$, $mm'n$, $m_1n_1m_2\ldots n_km_{k+1}$, $n_1m_1n_2\ldots m_kn_{k+1}$, with $m\neq m', m_i\in M$ and $n\neq n', n_i\in N$. These representations can be simplified even further, once two elements $e_M\in M$ and $e_N\in N$ are fixed, giving $m$, $n$, $mne_N$, $nme_M$, $mne_Me_N\ldots e_Ne_M$, $nme_Ne_M\ldots e_Me_N$. Colloquially, one refers to the alternating sequences of $e_M$ and $e_N$ appearing is such representations as to {\em tails}. By using this latter representation, one can show an important isomorphism of heaps which relates direct sums of heaps with direct sums of their retracts:
\begin{equation}\label{iso.direct}
M\boxplus N \cong \hH\left(\gG\left(M;e_M\right)\oplus \gG\left(N;e_N\right)\oplus \ZZ\right),
\end{equation}
in which tails are sent to integers; see \cite[Proposition~3.9]{BrzRyb:mod}. An immediate consequence of this isomorphism is that the direct sum of two non-empty abelian heaps is always an infinite abelian heap.

If $M$ and $N$ are left $T$-modules then the direct sum of heaps $M\boxplus N$ is a left $T$-module with action defined letter-wise on the (representative) reduced words. Its structure maps are the two $T$-linear morphisms $\iota_M:M\lra M \boxplus N, m\lto m,$ and $\iota_N: N \lra M\boxplus N, n \lto n$. The module structure induced on the groups on the right-hand side of isomorphism \eqref{iso.direct} is rather intricate, unless both $e_M$ and $e_N$ are absorbers.  In the latter case, it assumes the simple form $t \cdot (m, n, z) = (t \cdot m, t \cdot n, z)$ for all $t\in T$, $m \in M$, $n \in N$ and $z \in \ZZ$. 

For a (unital) truss $T$ and for every element $x$ of a set $X$, one can construct a (unital) left $T$-module $Tx$ generated by $x$. Namely,
$$
Tx\coloneqq \{tx\;|\;t\in T\}, \qquad [tx,t'x,t''x] \coloneqq  [t,t',t'']x, \quad t\cdot (t'x) = (tt')x,
$$ 
The map $T\lra Tx$, $t\lto tx,$ is an obvious isomorphism of modules. As explained in \cite[Section~4]{BrzRyb:mod}, for $T$ a unital truss the direct sum module 
$$
\Tt^{X}\coloneqq \underset{x\in X}{\boxplus}{Tx}.
$$
together with the function $\iota_X:X\lra \Tt^{X}$, $x\lto 1x,$ is a free object in the category of unital left $T$-modules.  Consequently, a unital left $T$-module $M$ is said to be a {\em free unital module generated by a set $X$} if it is isomorphic to $\Tt^{X}$.

\subsection{From trusses to unital trusses}\label{ssec.unital}

Recall from \cite[Proposition 3.12]{BrzRyb:mod} that the direct sum of heaps allows us to perform an analogue for trusses of the Dorroh's extension of a ring. Namely, for a truss $T$ the direct sum of abelian heaps $\uT \coloneqq T \boxplus \star$, with multiplication uniquely determined by relations
\begin{equation}\label{eq:unital}
{*}\cdot {*} = {*}, \qquad {*} \cdot t = t = t \cdot {*} \qquad \text{ and } \qquad t\cdot t' = tt'
\end{equation}
for all $t,t'\in T$, is a unital truss, called the \emph{unital extension} of $T$. It satisfies the following universal property.

\begin{proposition}\label{prop:unital}
Let $T$ be a truss. The heap homomorphism $\iota_T : T \lra \uT, t \lto t, $ is a morphism of trusses. Furthermore, if $S$ is a unital truss, then for every morphism of trusses $f : T \lra S$ there exists a unique morphism of unital trusses $\tilde{f} : \uT  \lra S$, such that $\tilde{f}\circ \iota_T = f$.
\end{proposition}

\begin{proof}
Since $\iota_T(t) = t$ for all $t\in T$, the right-hand side relation in \eqref{eq:unital} is exactly the multiplicativity of $\iota_T$. Assume that $S$ is a unital truss and that $f: T \lra S$ is a truss homomorphism. If we consider the heap map $\eta:\star \lra S, * \lto 1_S,$ then there exists a unique morphism of abelian heaps $\tilde{f}: T \boxplus \star \lra S$ such that $\tilde{f}\circ \iota_T = f$ and $\tilde{f} \circ \iota_\star = \eta$. We claim that $\tilde{f}$ is a morphism of unital trusses. Unitality follows by definition, since $\tilde{f}(*) = \eta(*) = 1_S$. To check multiplicativity pick two symmetric words $\sym{a_1a_2a_3\ldots a_{2k}a_{2k+1}}$ and $\sym{b_1b_2b_3\ldots b_{2h}b_{2h+1}}$ in $T \boxplus \star$, where the symbols $a_i,b_j$ belongs to $T \sqcup \star$ for all $i,j$. Since we have that
\[
\begin{gathered}
\tilde{f}\Big({\sym{a_1a_2a_3\ldots a_{2k}a_{2k+1}}} \cdot {\sym{b_1b_2b_3\ldots b_{2h}b_{2h+1}}}\Big) = \\
= \tilde{f}\Big(\sym{(a_1\cdot b_1)(a_1\cdot b_2)\ldots(a_1\cdot b_{2h+1})(a_2\cdot b_1)\ldots(a_i\cdot b_j)\ldots (a_{2k+1}\cdot b_{2h+1})}\Big) \\
= \Big[\tilde{f}(a_1\cdot b_1),\tilde{f}(a_1\cdot b_2),\cdots,\tilde{f}(a_1\cdot b_{2h+1}),\tilde{f}(a_2\cdot b_1),\cdots,\tilde{f}(a_i\cdot b_j),\cdots, \tilde{f}(a_{2k+1}\cdot b_{2h+1})\Big],
\end{gathered}
\]
it is enough to check that $\tilde{f}$ is multiplicative on a product $a \cdot b$ where $a,b \in T \sqcup \star$. Now,
\[
\tilde{f}(a\cdot b) = \begin{cases}
\tilde{f}({*} \cdot {*}) = \tilde{f}(*) = 1_S = 1_S \cdot 1_S = \tilde{f}(*) \cdot \tilde{f}(*) & a,b \in \star \\
\tilde{f}(t \cdot {*}) = \tilde{f}(t) = \tilde{f}(t)\cdot 1_S = \tilde{f}(t)\cdot \tilde{f}(*) & a\in T, b\in\star \\
\tilde{f}({*}\cdot t) = \tilde{f}(t) = 1_S \cdot \tilde{f}(t) = \tilde{f}(*) \cdot \tilde{f}(t) & a \in \star, b \in T \\
\tilde{f}(t \cdot t') = \tilde{f}(tt') = f(tt') = f(t)\cdot f(t') = \tilde{f}(t) \cdot \tilde{f}(t') & a,b \in T
\end{cases},
\]
that is, $\tilde{f}(a\cdot b) = \tilde{f}(a) \cdot \tilde{f}(b)$ for all $a,b \in T \sqcup \star$ and the proof is complete.
\end{proof}

\begin{theorem}\label{thm:unital}
Let $T$ be a truss. Any $T$-module $M$ is naturally a unital $\uT $-module. This induces a functor $\Uu : \lmod{T} \lra \lmod{(\uT)_1 }$ which is the inverse of the restriction of scalars functor $\iota_T^*:\lmod{(\uT)_1 } \lra \lmod{T}$ along the truss homomorphism $\iota_T : T \to \uT $. In particular, we have an isomorphism of categories $\lmod{(\uT)_1} \cong \lmod{T}$.
\end{theorem}

\begin{proof}
Observe that to equip  an abelian heap $M$ with the structure of a $T$-module is the same as to define a truss homomorphism $\rho_M:T \lra E(M)$. Since $E(M)$ is unital with unit $\id_M$, $\rho$ extends uniquely to a unital truss homomorphism $\varrho_M:\uT  \lra E(M)$ by Proposition \ref{prop:unital}, making of $M$ a unital $\uT $-module. Let $f:M \to N$ be a morphism of $T$-modules. To check that it is $\uT$-linear as well, observe that $f$ is $\uT$-linear if and only if the following diagram commutes
\[
\xymatrix @R=20pt{
M \ar[r]^-{\check{\varrho}_M} \ar[d]_-{f} & {\ahrd}\left({\uT},{M}\right) \ar[d]^-{{\ahrd}\left({\uT},{f}\right)} \\
N \ar[r]_-{\check{\varrho}_N} & {\ahrd}\left({\uT},{N}\right),
}
\]
where $\check{\varrho}_M(m): z \lto \varrho_M(z)(m)$ and ${\ahrd}\left({\uT},{f}\right): g \lto f \circ g$. Therefore, we are led to check that, for all $m \in M$,
\begin{equation}\label{eq:unitTech1}
f \circ \check{\varrho}_M(m) = \check{\varrho}_N(f(m))
\end{equation}
as heap homomorphisms from $\uT$ to $N$. However, since
\[
\begin{aligned}
\left(f \circ \check{\varrho}_M(m) \circ \iota_\star\right)(*) & = f \left( \check{\varrho}_M(m)(\iota_\star(*)) \right) = f \left( \varrho_M(\iota_\star(*))(m) \right) = f(m) \\
& = \varrho_N(\iota_\star(*))(f(m)) = \left(\check{\varrho}_N(f(m))\circ \iota_\star\right)(*) \qquad \text{ and} \\
\left(f \circ \check{\varrho}_M(m) \circ \iota_T\right)(\iota_T(t)) & = f \left( \check{\varrho}_M(m)(\iota_T(t)) \right) = f \left( \varrho_M(\iota_T(t))(m) \right) = f \left( \rho_M(t)(m) \right) \\
& = \rho_N(t)\left(f(m)\right) = \varrho_N(\iota_T(t))(f(m)) = \left(\check{\varrho}_N(f(m))\circ \iota_T\right)(t),
\end{aligned}
\]
for all $t \in T$, it follows by the universal property of the direct sum that \eqref{eq:unitTech1} holds. Summing up, there is a fully faithful functor
\[
\Uu: \lmod{T} \lra \lmod{(\uT)_1}, \qquad  \begin{cases} M \lto M \\ f \lto f \end{cases}.
\]
Now, if $(M,\rho_M)$ is a $T$-module and we consider its unital extension $\left(\Uu(M),\varrho_M\right)$, then the restriction of scalars functor $\iota_T^*$ endows $\Uu(M)$ with the $T$-module structure given by the composition
\[
T \xrightarrow{\iota_T} \uT \xrightarrow{\varrho_M} E(M)
\]
which coincides with $\rho_M$ by definition of $\varrho_M$. The other way around, if $(N,\varrho_N)$ is a unital $\uT$-module and we construct the unital extension $\Uu(\iota_T^*(N))$ of the $T$-module $(\iota_T^*(N),\rho_N)$ obtained by restriction of scalars along $\iota_T$, then this is given by the unique unital extension of $\rho_N = \varrho_N\circ \iota_T$ and the latter has to coincide with $\varrho_N$ by uniqueness.
\end{proof}

 As a matter of notation, we will often omit to specify the functors $\iota_T^*$ and $\Uu$, unless their presence would increase the clarity of the exposition.

\subsection{Modules over a ring and modules over its truss}\label{ssec.ringtype}
Recall from \cite[\S2.2 and \S4]{BrzRyb:mod} that if $R$ is a (unital) ring then we can consider its associated (unital) truss $\tT(R) = (\hH(R,+),\cdot)$. Moreover, any (unital) $R$-module $M$ gives rise, in the same way, to a (unital) $\tT(R)$-module $\tT(M)=(\hH(M,+),\cdot)$, whose underlying abelian heap structure is induced by the abelian group one. This assignment gives rise to a functor
\[
\tT : \pMod{R} \lra \pMod{\tT(R)}, \quad (M,+,\cdot) \lto (\hH(M,+),\cdot), \quad f \lto  f,
\]
which admits a left adjoint
\[
(-)_{\Abs} : \pMod{\tT(R)} \lra \pMod{R}, \;\; (M,[\mbox{-},\mbox{-},\mbox{-}],\cdot) \lto \left(\gG\left(M/\Abs{(M)};\Abs{(M)}\right),\cdot\right) .
\]
In view of \cite[Lemma 4.6(5)]{BrzRyb:mod}, the counit $\eps_N: \tT(N)_{\Abs} \lra N, \bar{n} \lto n,$ of this adjunction is always a natural isomorphism and hence $\tT$ is fully faithful (see \cite[Theorem IV.3.1]{MacLane}). The unit $\eta_M: M \lra \tT(M_\Abs), m \lto \bar{m},$ is simply the canonical projection onto the quotient $M/\Abs(M)$, for all $M$ in $\pMod{\tT(R)}$.

\subsection{\for{toc}{Epis, monos and coequalizers of modules}\except{toc}{Epimorphisms, monomorphisms and coequalizers of \texorpdfstring{$T$}{T}-modules}}\label{sec.cat} Let $T$ be a truss. It will be useful in the forthcoming sections to know that epimorphisms (respectively, monomorphisms) of $T$-modules are always  \emph{effective}, \ie that they are coequalizers (respectively, equalizers) of  their kernel pairs (respectively, cokernel pairs), and that they coincide with surjective (respectively, injective) $T$-linear maps.

 To this aim, recall that if $f:M \lra N$ is a morphism of $T$-modules, its kernel pair (respectively, cokernel pair) is the pullback (respectively, pushout) of the pair $(f,f)$.

\begin{proposition}\label{prop:episurj}
Every epimorphism of $T$-modules is surjective.
\end{proposition}

\begin{proof}
Assume that $M$ and $N$ are $T$-modules. If both $M$ and $N$ the empty $T$-modules, the empty map is an epimorphism (by uniqueness) and it is also surjective (trivially). If only $N$ is the empty module, then we cannot have morphisms from a non-empty to the empty module. If only $M$ is the empty module, then the empty map to $N$ is not an epimorphism. Summing up, we may assume that both $M$ and $N$ are non-empty and that  $\varphi:M\lra N$ is an epimorphism of $T$-modules. Consider the $T$-submodule $\im (\varphi)\subseteq N$ and the canonical projection $\pi: N \lra N/\im (\varphi)$. For $p\in \im (\varphi)$, consider also the constant morphism $\tau_{\bar{p}}: N \lra N/\im (\varphi)$, $n\lto \bar{p}\coloneqq \pi(p)$.

For every $m\in M$, $\varphi(m)\sim_{\im (\varphi)} p$ and hence $\pi(\varphi(m)) = \pi(p) = \bar{p} = \tau_{\bar{p}}(\varphi(m))$. Since $\varphi$ is an epimorphism,  $\pi=\tau_{\bar{p}}$ and hence  every $n\in N$ satisfies $n \sim_{\im (\varphi)} p$ (\ie for all $\varphi(m)\in \im (\varphi)$, $[n,p,\varphi(m)]\in \im (\varphi)$). In particular, $n=[n,p,p]\in \im (\varphi)$ for all $n\in N$ and so $\varphi$ is surjective.
\end{proof}

\begin{proposition}\label{prop:epicoeq}
Every epimorphism of $T$-modules is the coequalizer of its kernel pair.
\end{proposition}

\begin{proof}
Assume that $\pi:M\lra P$ is an epimorphism of $T$-modules. The kernel relation  together with its coordinate projections
$$
\begin{gathered}
\Ker(\pi) = \{ (m_1,m_2)\in M\times M\; |\;\; \pi(m_1)=\pi(m_2)\}\subseteq M\times M,\\
p_i: \Ker(\pi)\lra M, \qquad (m_1,m_2)\lto m_i,  \qquad i=1,2,
\end{gathered}
$$ 
yield the following fork of $T$-modules
\[
\xymatrix{
\Ker(\pi) \ar@<+0.5ex>[r]^-{p_1} \ar@<-0.5ex>[r]_-{p_2} & M \ar[r]^-{\pi} & P.
}
\]
Assume that $f:M\lra N$ is any other $T$-module map such that $f\circ p_1 = f\circ p_2$ and consider $\bar{f}:P\lra N$ given by $\bar{f}(\pi(m))\coloneqq  f(m)$. The map $\bar{f}$ is well-defined because if $\pi(m_1) = \pi(m_2)$, then $(m_1,m_2)\in \Ker(\pi)$ and hence $f(m_1) = (f\circ p_1) (m_1,m_2) = (f\circ p_2) (m_1,m_2) = f(m_2)$. It is a morphism of $T$-modules because $\pi$ and $f$ are $T$-linear maps. It is a unique morphism such that $\bar{f}\circ \pi = f$ because $\pi$ is an epimorphism. Thus, $(P,\pi)$ satisfies the universal property of the coequalizer of the pair $(p_1,p_2)$. To conclude, observe that $\left(\Ker(\pi),p_1,p_2\right)$ is the kernel pair of $f$.
\end{proof}

\begin{proposition}\label{prop:monoinj}
Every monomorphism of $T$-modules is injective.
\end{proposition}

\begin{proof}
Let $f: M \lra N$ be a monomorphism of $T$-modules. As before, there is a fork diagram of $T$-modules
\[
\xymatrix{
\Ker(f) \ar@<+0.5ex>[r]^-{p_1} \ar@<-0.5ex>[r]_-{p_2} & M \ar[r]^-{f} & N.
}
\]
The fact that $f$ is a monomorphism implies that $p_1 = p_2$ and hence $(m,n) \in \Ker(f)$ if and only if $m=n$, which in turn entails that $f(m) = f(n)$ if and only if $m = n$.
\end{proof}

\begin{lemma}\label{lem:boxplusdiff}
Let $M,N$ be $T$-modules, $M \boxplus N$ their coproduct in $\pMod{T}$ and $\iota_M: M \lra M \boxplus N$, $\iota_N:N \lra M\boxplus N$ the structure maps of the coproduct. Then $\iota_M(m) \neq \iota_N(n)$ for all $m \in M$, $n \in N$.
\end{lemma}

\begin{proof}
Endow the abelian heap $\hH(\ZZ_2)$ with the trivial $T$-module structure: $t\cdot x = x$ for all $t \in T$ and $x \in \ZZ_2$. The assignments
\[
\varphi_M : M \lra \hH(\ZZ_2),\ m \lto 0, \qquad \text{ and } \qquad \varphi_N: N \lra \hH(\ZZ_2),\ n \lto 1,
\]
are well-defined $T$-linear morphisms and hence they induce, by the universal property of the coproduct, a unique $T$-linear map $\Phi: M \boxplus N \lra \hH(\ZZ_2)$ such that $\Phi\circ \iota_M = \varphi_M$ and $\Phi \circ \iota_N = \varphi_N$. If we suppose that there exist $m \in M$ and $n \in N$ such that $\iota_M(m) = \iota_N(n)$, then
\[
0 = \varphi_M(m) = \Phi(\iota_M(m)) = \Phi(\iota_N(n)) = \varphi_N(n) = 1,
\]
which is a contradiction.
\end{proof}

\begin{proposition}\label{prop:monoeq}
Every monomorphism of $T$-modules is  the equalizer of its cokernel pair.
\end{proposition}

\begin{proof}
 Since the category of $T$-modules is cocomplete (by \cite[Theorem 9.3.8]{Ber:inv}, for example), it is enough to prove that every monomorphism is regular, that is, that it is the equalizer of some pair of arrows.

Assume that $M$ and $N$ are $T$-modules. If $M$ is the empty $T$-module, then the empty map is a monomorphism (because there are no maps from a non-empty to the empty module) and it is also the equalizer of the pair
\[
\xymatrix @R=0pt{
N \ar[rr]^-{\iota_N} \ar[dr]_-{*} & & N \boxplus \star, \\
  & \star \ar[ur]_-{\iota_\star} &
}
\]
by Lemma \ref{lem:boxplusdiff}. If $M$ is non-empty, then $N$ cannot be the empty module, since we cannot have morphisms from a non-empty to the empty module. Summing up, we may assume that both $M$ and $N$ are non-empty and that  $f:M\lra N$ is a monomorphism of $T$-modules. Consider then $e' \in M$, $N\supseteq \im (f) \ni e = f(e')$, the quotient $T$-module $N/\im (f)$, the absorber $\overline{e}=\im (f)$ therein and the canonical projection $\pi: N \lra N/\im(f)$. Then there is a fork diagram of $T$-modules
\begin{equation}\label{eq:equalizer}
\xymatrix{
M \ar[r]^-{f} & N \ar@<+0.4ex>[r]^-{\pi} \ar@<-0.4ex>[r]_-{\tau_{\overline{e}}} & N/\im(f),
}
\end{equation}
where $\tau_{\overline{e}}$ denotes the $T$-linear morphism sending everything to $\overline{e}$. Let us check that $(M,f)$ is the equalizer of the pair $(\pi,\tau_{\overline{e}})$. 
If $P$ is another $T$-module and $g : P \lra N$ is a $T$-linear map such that $\pi(g(p)) = \overline{e}$ for all $p \in P$, then this implies that there exists $f(m) \in \im(f)$ such that
\[
[g(p),f(e'),f(m)] = [g(p),e,f(m)] \in \im(f).
\]
In particular,
\[
g(p) = \left[\left[ g(p),f(e'),f(m) \right],f(m),f(e')\right] \in \im(f),
\]
and hence there exists a (necessarily unique, in view of Proposition \ref{prop:monoinj}) element $m_p \in M$ such that $g(p) = f(m_p)$. Since, in addition,
\[
f(m_{t \cdot p}) = g(t \cdot p) = t\cdot g(p) = t\cdot f(m_p) = f(t \cdot m_p),
\]
for all $p \in P$ and $t \in T$, the assignment $h:P \lra M, p \lto m_p,$ is a $T$-linear morphism such that $f\circ h = g$ and it is unique satisfying this property, because $f$ is injective. Summing up, $(M,f)$ is indeed the equalizer of \eqref{eq:equalizer}, as claimed.
\end{proof}

Finally, since equalizers of abelian heaps and $T$-modules are simply equalizers in $\Set$ endowed with the sub-heap or $T$-submodule structure, let us describe explicitly a construction of coequalizers in the categories of abelian heaps and $T$-modules.

\begin{lemma}\label{lem.coeq}
Given a diagram 
\begin{equation}\label{coeq}
\xymatrix @C=40pt{A\ar@<.6ex>[r]^\varphi\ar@<-.6ex>[r]_\psi & B}
\end{equation}
in $\ahrd$ and any $e\in B$, define
\begin{equation}\label{ne}
N(e) = \{[\varphi(a), \psi(a), e] \mid a\in A\}.
\end{equation}
Then
\begin{enumerate}[label=(\arabic*),ref=\textit{(\arabic*)},leftmargin=0.8cm]
\item\label{coeq:item1} The set $N(e)$ is a sub-heap of $B$ and, for different choices of $e$, the heaps $N(e)$ are mutually isomorphic.
\item\label{coeq:item2} Let $\overline{N(e)} \coloneqq  \langle N(e),e\rangle$ be the sub-heap of $B$ generated by $N(e)$ and $e$. 
The quotient heap $C(e) = B/\overline{N(e)}$ is the coequalizer of \eqref{coeq}.
\item\label{coeq:item3} If \eqref{coeq} is a diagram in $T$-$\Mod$, where $T$ is a truss, then $C(e)$ is its coequalizer in $T$-$\Mod$.
\end{enumerate}
\end{lemma}
\begin{proof}
\ref{coeq:item1} That $N({e})$ is a sub-heap of $B$ follows by $\eqref{tran}$ and the fact that $\varphi,\psi$ are morphisms of heaps. Let $f\in B$. The isomorphism between $N({e})$ and $N({f})$ is given by $\tau_e^f$ of \eqref{swap}. 

\ref{coeq:item2} 
Let us check that the canonical projection $\pi: B\lra C(e) = B/\overline{N(e)}$ coequalizes $\varphi$ and $\psi$. Since $e\in \overline{N(e)}$ and $[\varphi(a),\psi(a),e]\in \overline{N(e)}$,  $\varphi(a) \sim_{\overline{N(e)}}\psi(a)$, and hence $\pi(\varphi(a))=\pi(\psi(a)).$ Therefore, there is the required fork
$$
\xymatrix @C=40pt{A\ar@<.6ex>[r]^\varphi\ar@<-.6ex>[r]_\psi & B \ar[r]^-{\pi}& C(e).}
$$
Now, let us assume that there exists another pair $(h,H)$ such that $h:B\rightarrow H$ and $h\circ \varphi=h\circ\psi$. Observe that, for all $a\in A,$  
$$h([\varphi(a),\psi(a),e])=[h(\varphi(a)),h(\psi(a)),h(e)]=h(e),$$
where the second equality follows from $h\circ \varphi=h\circ\psi$ and Mal'cev identity. 
Thus, $h(x)=h(e)$ for all $x\in \overline{N(e)}$ and so $\overline{N(e)}\subseteq \ker_{h(e)}(h)$. In view of Lemma~\ref{lem:quomap}, there is a unique  heap homomorphism $f:C({e})\rightarrow H$ given by $f(\pi(b))=h(b)$ for all $b\in B$.

\ref{coeq:item3} To prove that $C({e})$ is a coequalizer in the category of modules it is enough to prove that $\overline{N(e)}$ is an induced $T$-submodule. Since
\[
\begin{aligned}
t\cdot_{e}[\varphi(a),\psi(a),e]&=[t\cdot [\varphi(a),\psi(a),e],t\cdot e,e] =[[t\cdot \varphi(a),t\cdot \psi(a),t\cdot e],t\cdot e,e] \\
&=[t \cdot\varphi(a),t\cdot \psi(a),[t\cdot e,t\cdot e,e]] =[t\cdot\varphi(a),t\cdot\psi(a),e]\\
&=[\varphi(ta),\psi(ta),e]\in\overline{N(e)}
\end{aligned}
\]
and $t \cdot_e e = e$, it follows that $C({e})$ is a well-defined quotient module and the proof that $C({e})$ is the coequalizer is analogous to \ref{coeq:item2}.
\end{proof}

\section{Tensor product of modules over a truss}\label{sec.tens}
In this section we construct the tensor product between modules over a truss and study its categorical properties. 

\subsection{The definition and construction of the tensor product}\label{ssec.tensprod}
The aim of this section is to define and show the existence of a tensor product of modules over a truss.
\begin{definition}\label{def.balanced}
Let  $H,M,N$ be abelian heaps. A function $\varphi:M\times N\lra H$ is said to be {\em bilinear} if, for all $m,m',m'' \in M$ and $n,n',n''\in N$,
\begin{subequations}\label{balance}
\begin{equation}\label{lin.l}
 \varphi([m,m',m''],n)=[\varphi(m,n),\varphi(m',n),\varphi(m'',n)],
 \end{equation}
 \begin{equation}\label{lin.r}
 \varphi(m,[n,n',n''])=[\varphi(m,n),\varphi(m,n'),\varphi(m,n'')].
 \end{equation}
\end{subequations}
In addition, if $T$ is a truss, $M$ is a right $T$-module and $N$ is a left $T$-module, then $\varphi$ is said to be {\em $T$-balanced} if, for all $m\in M$, $n \in N$, $t\in T$, 
\begin{subequations}[resume]
 \begin{equation}\label{t.balance} 
\varphi(m\cdot t,n)=\varphi(m,t\cdot n).
\end{equation}
\end{subequations}
\end{definition}

\begin{remark}\label{rem:heapisbilinear}
We note in passing that, due to the Mal'cev conditions, any heap homomorphism $\varphi: M\times N\lra H$ satisfies conditions \eqref{lin.l}--\eqref{lin.r} in Definition~\ref{def.balanced} (but, of course, a function satisfying \eqref{lin.l}--\eqref{lin.r} needs not be a homomorphism of heaps).
\end{remark}

The definition of the tensor product of modules over a truss is given by the  following universal property, reminiscent of that for the tensor product of modules over a ring.
\begin{definition}\label{tensor}
Let $M$ be a right $T$-module and $N$ be a left $T$-module. Then  a {\em tensor product} (of $M$ and $N$ over $T$) is a pair $(M\otimes_{T} N,\varphi)$ consisting of an abelian heap $M\otimes_{T} N$ and a $T$-balanced bilinear map $\varphi:M\times N\lra M\otimes_{T} N$ such that for any heap $H$ and any $T$-balanced bilinear map $f:M\times N\lra H$ there exists a unique heap morphism $\hat{f}$ rendering commutative the following diagram 
$$
\xymatrix @R=20pt{M\times N \ar[rr]^-{\varphi} \ar[dr]_-f&&  M\otimes_{T} N\ar@{-->}[dl]^-{\exists !\, \hat{f}}
\\
& H. &}
$$
\end{definition}

 As for tensor products of modules over rings, if a tensor product of $M$ and $N$ over $T$ exists, then it is unique up to a unique isomorphism. Thus, we will speak about \emph{the} tensor product $M\otimes_{T} N$, often omitting the structure map $\varphi$ as well.

Since any abelian heap is a unital module over the terminal truss $\star$ in a unique way, one can consider tensor product of heaps. In this case the balancing condition \eqref{t.balance} is tautologically satisfied. The tensor product of heaps $M$ and $N$ viewed as unital $\star$-modules is denoted by $M\otimes N$. Observe that, differently from what happens for modules over a ring, the fact that $\id_{M\times N}$ is bilinear entails that there exists a unique morphism of heaps $\sigma : M \otimes N \to M \times N$ such that $\sigma \circ \varphi = \id_{M\times N}$ (see Remark \ref{rem:heapisbilinear}).

Next we give an explicit construction of tensor products, thus establishing their existence.

\begin{theorem}\label{thm.tensor}
Tensor product of $T$-modules exists.
\end{theorem}

\begin{proof}
Let $M$ be a right $T$-module and $N$ be a left $T$-module. If $M$ or $N$ is empty, then $M \ot_T N$ is empty. Otherwise, let us consider the free abelian heap $\Aa(M\times N)$. Choose an arbitrary element $e=(e_{1},e_{2})$ of the free heap $\Aa(M\times N)$  and let  $S_T({e})$ be the sub-heap  of  $\Aa(M\times N)$ generated  by elements of the form:
\begin{subequations}\label{elements}
\begin{equation}\label{element1}
\Big[\big([m,m',m'']_M,n\big),\big[(m,n),(m',n),(m'',n)\big]_{\Aa},e\Big]_{\Aa},
\end{equation}
\begin{equation}\label{element2}
\Big[\big(m,[n,n',n'']_N\big),\big[(m,n),(m,n'),(m,n'')\big]_{\Aa},e\Big]_{\Aa},
\end{equation}
\begin{equation}\label{element3}
\big[(m\cdot t,n), (m,t \cdot n), e \big]_{\Aa},
\end{equation}
\end{subequations}
for all $m,m',m''\in M,$ $n,n',n''\in N$ and $t\in T$. Note that the transposition rule \eqref{tran} together with the idempotent property of a heap operation  imply  that every element of $S_T({e})$  has the form $[a,b,e]$, where $a,b\in \Aa(M\times N)$. Also note that $e\in S_T(e)$.
For an abelian heap $H$, consider a $T$-balanced bilinear map $f :M\times N\lra H$. By treating $f$ as a function and by using the universal property of the free heap, we can construct the following commutative diagram:
\begin{equation}\label{graf1}
\xymatrix @R=20pt{M\times N \ar[r]^-{\iota_{M\times N}} \ar[dr]_-f& \Aa(M\times N)  \ar@{-->}[d]^-{\exists !\, \hat{f}}\ar[r]^-{\pi_{S_T(e)}}& \Aa(M\times N)/S_T(e)\ar@{-->}[dl]^-{\exists !\, \hat{\hat{f}}}
\\
& H, &}
\end{equation}
where $\iota_{M\times N}$ is the canonical monomorphism and $\pi_{S_T(e)}$ is the canonical epimorphism. The left triangle is given by the free heap property.  The existence of the map $\hat{\hat{f}}$ is guaranteed  provided that $\hat{f}$ respects the sub-heap relation $\sim_ {S_T(e)}$. By using the definition of $\hat{f}$ and that $f$ is a $T$-balanced bilinear map, we find
$$
\begin{aligned}
\hat{f}\bigg(\Big[\big(m,[n,n',n'']_N\big),& \big[(m,n),(m,n'),(m,n'')\big]_{\Aa},e\Big]_{\Aa}\bigg) \\
&\stackrel{\phantom{\eqref{lin.r}}}{=}\bigg[f\Big(m,[n,n',n'']_N\Big),\big[f(m,n),f(m,n'),f(m,n'')\big]_{H},f(e)\bigg]_{H} \\ 
&\stackrel{\eqref{lin.r}}{=}\Big[f\big(m,[n,n',n'']_N\big),f\big(m,[n,n',n'']_N\big),f(e)\Big]_{H}=f(e)=\hat{f}(e).
\end{aligned}
$$
By symmetric arguments,
$$
\begin{aligned}
\hat{f}\bigg(\Big[\big([m,m',m'']_M,n\big), \big[(m,n),(m',n),(m'',n)\big]_{\Aa},e\Big]_{\Aa}\bigg)
=\hat{f}(e).
\end{aligned}
$$
Finally,
$$
\begin{aligned}
 \hat{f}\big([(m\cdot t,n),(m,t\cdot n),e]_{\Aa}\big)& =\big[f(m\cdot t,n),f(m,t\cdot n),f(e)\big]_{H} \\ &= \big[f(m\cdot t,n),f(m\cdot t,n), f(e)\big]=f(e)=\hat{f}(e).
\end{aligned}
$$
This means that $S_T(e)\subset \ker_{\hat{f}(e)}(\hat{f})$ and therefore, in view of Lemma \ref{lem:quomap}, $\hat{f}$ respects the sub-heap relation $\sim_ {S_T(e)}$ as required. Consequently, the  heap homomorphism $\hat{\hat{f}}$ exists. 

Define: 
$$
\varphi \coloneqq \left(\pi_{S_T(e)}\circ\iota_{M\times N}\right): M\times N \longrightarrow  \Aa(M\times N)/S_T(e), \qquad (m,n) \lto \overline{(m,n)}.
$$
Since $e \in S_T(e)$, by definition of $\sim_{S_T(e)}$ and of $[-,-,-]$ on $\Aa(M\times N)/S_T(e)$,
\[
\begin{gathered}
\overline{\big([m,m',m'']_M,n\big)} = \overline{\big[(m,n),(m',n),(m'',n)\big]_{\Aa}} = \big[\overline{(m,n)},\overline{(m',n)},\overline{(m'',n)}\big], \\
\overline{\big(m,[n,n',n'']_N\big)} = \overline{\big[(m,n),(m,n'),(m,n'')\big]_{\Aa}} = \big[\overline{(m,n)},\overline{(m,n')},\overline{(m,n'')}\big], \\
\text{and} \qquad \overline{(m\cdot t,n)} = \overline{(m,t \cdot n)}
\end{gathered}
\]
hold in $\Aa(M\times N)/S_T(e)$, that is to say, $\varphi$ is a $T$-balanced bilinear map.
It remains to prove that the map $\hat{\hat{f}}$ constructed in diagram \eqref{graf1} is a unique homomorphism such that $f= \hat{\hat{f}}\circ \pi_{S_T(e)}\circ\iota_{M\times N}$. Suppose that there exists another homomorphism of heaps $h: \Aa(M\times N)/S_T(e) \lra H$ such that $f= h\circ \pi_{S_T(e)}\circ\iota_{M\times N}$. Then 
$$
h\circ \pi_{S_T(e)}\circ\iota_{M\times N} = \hat{\hat{f}}\circ \pi_{S_T(e)}\circ\iota_{M\times N},
$$
and, since both  $h\circ \pi_{S_T(e)}$ and $\hat{\hat{f}}\circ \pi_{S_T(e)}$ are homomorphisms of heaps, the universal property of the free (abelian) heap implies that 
$$
h\circ \pi_{S_T(e)} = \hat{\hat{f}}\circ \pi_{S_T(e)}.
$$
Since $\pi_{S_T(e)}$ is an epimorphism, it follows that $h=\hat{\hat{f}}$ and the uniqueness is established. 
Therefore,  the pair $(\Aa(M\times N)/S_T(e),\varphi)$ is the tensor product of $M$ and $N$.
\end{proof}
 
  We note in passing that, up to isomorphism, the construction of the tensor product does not depend on the chosen element $e$. This independence can be seen as a consequence of the universal property of tensor products, or it can be observed directly by employing the swap automorphism \eqref{swap}.

Following the ring-theoretic conventions we define, for all $m\in M$ and $n\in N$, 
\begin{equation}\label{tens.simp}
m\ot n \coloneqq \overline{(m,n)} = \left(\pi_{S_T(e)}\circ\iota_{M\times N}\right)(m,n)\in M\ot_T N
\end{equation}
and we refer to each of $m\ot n$ as to a {\em simple tensor}. As a rule, we do not decorate $\ot$ with a subscript $T$, but occasionally it might be useful to indicate an element $e$ chosen in the definition of $S_T(e)$, in which case we write $m\ot_en$.  With this terminology and notation at hand, $M\ot_TN$ can be understood as an abelian heap freely generated by simple tensors subject to relations:
\begin{subequations}
\begin{equation}\label{simp.l}
[m,m',m'']\ot n = [m\ot n, m' \ot n, m''\ot n], \qquad \mbox{for all $m,m',m''\in M$, $n\in N$},
\end{equation}
\begin{equation}\label{simp.r}
m\ot [n,n',n''] = [m\ot n, m \ot n', m\ot n''], \qquad \mbox{for all $m\in M$, $n,n',n''\in N$},
\end{equation}
\begin{equation}\label{simp.b}
m\cdot t\ot n = m\ot t\cdot n, \qquad \mbox{for all $m\in M$, $n\in N$, $t\in T$}.
\end{equation}
\end{subequations}

We conclude the subsection with a technical result that will be of significant importance in \S\ref{sec.morita}.
\begin{proposition}\label{prop:tensunital}
Let $T$ be a truss and $\uT$ be its unital extension as in \S\ref{ssec.unital}. Then for every right $T$-module $M$, $M \otimes_T \uT \cong M \cong \Hom_T\left(\uT,M\right)$ as right $T$-modules, where $\uT$ has the $T$-$T$-bimodule structure induced by the truss homomorphism $\iota_T : T \to \uT$. Moreover, for $M$ a right $T$-module and $N$ a left $T$-module
\[
m\cdot z \otimes_T n = m \otimes_T z \cdot n,
\]
for all $m \in M$, $n \in N$, $z \in \uT$.
\end{proposition}

\begin{proof}
Consider the assignment
\[
\alpha : M \lra M \otimes_T \uT, \qquad m \lto m \otimes_T {*}.
\]
This is a heap homomorphism in view of \eqref{simp.l} and it is $T$-linear because for all $t \in T$,
\[
\alpha(m \cdot t) = m \cdot t \otimes_T {*} \stackrel{\eqref{simp.b}}{=} m \otimes_T \iota_T(t)\cdot {*} \stackrel{\eqref{eq:unital}}{=} m \otimes_T {*} \cdot \iota_T(t) = (m \otimes_T {*}) \cdot t.
\]
The other way around, recall that the underlying abelian heap of $\Uu(M)$ is $M$ itself, which now is considered as a unital $\uT$-module via the bilinear morphism $\varrho : M \times \uT \lra M$ uniquely determined by
\[
\begin{cases} (m,t) \lto m \cdot t, & \mbox{$t \in T $},\\ (m,{*}) \lto m .\end{cases}
\]
The associativity of the $T$-action entails that $\varrho$ is also $T$-balanced, whence it factors through the tensor product over $T$ giving
\[
\beta : M \otimes_T \uT \lra M.
\]
A straightforward check shows that $\alpha$ and $\beta$ are inverses of each other. Concerning the second isomorphism, consider the right $T$-linear morphism
\[
\Hom_T(\uT,M) \lra M, \qquad f \lto f(*),
\]
and the assignment $M \lra \Hom_T(\uT,M)$, sending every $m \in M$ to the right $T$-linear morphism uniquely determined by
\[
\begin{cases}
t \lto m \cdot t, & t \in T,\\
* \lto m. &
\end{cases}
\]
Again, a straightforward check shows that they are inverses of each other. To prove  the last assertion, recall that an element $z$ in $\uT$ is of the form $[a_1,\ldots,a_{s}]$, where $a_i \in T \sqcup \star$, for all $i=1,\ldots,s$ and $s$ odd. Therefore,
\[
\begin{aligned}
m\cdot z \otimes_T n &= m \cdot [a_1,\ldots,a_{s}] \otimes_T n = [m \cdot a_1\otimes_T n,\ldots,m \cdot a_{s}\otimes_T n] \\
&\stackrel{(\bullet)}{=} [m\otimes_T a_1 \cdot n,\ldots,m\otimes_T a_{s} \cdot n] = m \otimes_T z \cdot n,
\end{aligned}
\]
where $(\bullet)$ follows from the fact that either $m \cdot a_i \otimes_T n = m \cdot t \otimes_T n = m \otimes_T t \cdot n$ (if $a_i \in T$) or $m \cdot a_i \otimes_T n = m \otimes_T n = m \otimes_T a_i \cdot n$ (if $a_i = * \in \star$).
\end{proof}

\subsection{Functorial properties of tensor products}\label{ssec.functorial}
In parallel to the ring-theoretic tensor product, tensoring with a fixed bimodule defines a functor between categories of modules over trusses.

\begin{proposition}\label{prop.tens.act}
Let $T$ and $R$ be trusses. 
\begin{enumerate}[label=(\arabic*),ref=\textit{(\arabic*)},leftmargin=0.8cm]

\item\label{tens:item1} If $M$ is a right $T$-module and $N$ is a $T$-$R$-bimodule, then $M\otimes_T N$ is a right $R$-module with the action
$$
(M\ot_T N ) \times R\lra M\ot_T N , \qquad (m\ot n, r) \lto m\ot n\cdot r.
$$
 If $R$ admits a unit and $N$ is unital, then $M \ot_T N$ is unital as well. Symmetrically, if $M$ is an $R$-$T$-bimodule  (unital over $R$) and $N$ is a left  $T$-module, then  $M\otimes_T N$ is a  (unital) left $R$-module.

 \item\label{tens:item2} Let $N$ be a $T$-$R$-bimodule and let $\varphi: M\lra M'$ be a homomorphism of right $T$-modules. Then the map  $\varphi\ot N$ defined on simple tensors as
 $$
 \varphi\ot N : M\ot_T N\lto M'\ot_T N, \qquad m\ot n\lto \varphi(m)\ot n,
 $$
 extends uniquely to a homomorphism of right $R$-modules. Symmetrically, if $M$ is an $R$-$T$ bimodule, then  any left $T$-module homomorphism $\varphi: N \lra N'$ gives rise to a left $R$-module homomorphism,
 $$
 M\ot \varphi : M\ot_T N\lto M\ot_T N', \qquad m\ot n \lto m\ot \varphi(n).
 $$

 \item\label{tens:item3} The constructions in items (1) and (2) yield functors $-\ot_T N: \rmod{T} \lra \rmod{R}$ and $M\ot_T - : \lmod{T} \lra \lmod{R}$. Furthermore, if $R$ admits a unit and $M,N$ are unital (over $R$), then they yield functors $-\ot_T N: \rmod{T} \lra \rmod{R_1}$ and $M\ot_T - : \lmod{T} \lra \lmod{R_1}$.
 \end{enumerate}
 \end{proposition}

 \begin{proof}
\begin{description}[leftmargin=0.3cm,font=\normalfont,itemindent=0cm]
\item[\ref{tens:item1}]  Since $N$ is a right $R$-module, for every $r\in R$ we can consider the assignment
\[
\rho_r:M \times N \lra M \otimes_T N, \qquad (m,n) \lto m\otimes n\cdot r.
\]
It satisfies
\[
\begin{aligned}
\rho_r(([m,m,',m''],n)) & \stackrel{\phantom{\eqref{simp.r}}}{=} [m,m,',m''] \otimes n\cdot r \stackrel{\eqref{simp.l}}{=}  \left[m\otimes n\cdot r, m'\otimes n\cdot r, m''\otimes n\cdot r\right] \\
 & \stackrel{\phantom{\eqref{simp.r}}}{=} \left[\rho_r(m, n), \rho_r(m', n), \rho_r(m'' , n)\right], \\
\rho_r(m,[n,n',n'']) & \stackrel{\phantom{\eqref{simp.r}}}{=} m \otimes [n,n',n'']\cdot r = m \otimes [n\cdot r,n'\cdot r,n''\cdot r] \\
 & \stackrel{\eqref{simp.r}}{=}[m \otimes n\cdot r,m \otimes n'\cdot r,m \otimes n''\cdot r] \\
 & \stackrel{\phantom{\eqref{simp.r}}}{=} \left[\rho_r(m, n), \rho_r(m, n'), \rho_r(m, n'')\right], \\
\rho_r((m\cdot t,n)) & \stackrel{\phantom{\eqref{simp.r}}}{=} m\cdot t \otimes n\cdot r \stackrel{\eqref{simp.b}}{=} m \otimes t\cdot (n\cdot r) = m \otimes (t\cdot n)\cdot r = \rho_r(m,t\cdot n),
\end{aligned} 
\]
for all $m,m',m''\in M$, $n,n',n'' \in N,$ $t \in T$. That is to say, $\rho_r$ is a $T$-balanced bilinear map and hence it factors uniquely through $M \otimes_T N$ via the heap morphism
\[
\varrho_r : M \otimes_T N \lra M\otimes_T N, \qquad m\otimes n \lto m\otimes n\cdot r.
\]
Now, consider the assignment
\[
\varrho:R \lra E(M \otimes_T N), \qquad r \lto \varrho_r.
\]
For all $m\in M$, $n\in N$, $r,r',r''\in R$, 
\[
\begin{aligned}
\varrho_{[r,r',r'']}(m \otimes n) & \stackrel{\phantom{\eqref{simp.r}}}{=} m \otimes n\cdot [r,r',r''] \stackrel{\eqref{module2}}{=} m \otimes [n\cdot r,n\cdot r',n\cdot r''] \\
 & \stackrel{\eqref{simp.r}}{=}[m \otimes n\cdot r,m \otimes n\cdot r',m \otimes n\cdot r''] \\
 & \stackrel{\phantom{\eqref{simp.r}}}{=}  \left[\varrho_{r}(m \otimes n), \varrho_{r'}(m \otimes n), \varrho_{r''}(m \otimes n)\right] \\
 & \stackrel{\phantom{\eqref{simp.r}}}{=} \left[\varrho_{r}, \varrho_{r'}, \varrho_{r''}\right](m \otimes n), \\
\varrho_{rr'}(m \otimes n) & \stackrel{\phantom{\eqref{simp.r}}}{=} m \otimes n\cdot rr' = m \otimes (n\cdot r)\cdot r' = \varrho_{r'}(m \otimes (n\cdot r)) \\
 & \stackrel{\phantom{\eqref{simp.r}}}{=} (\varrho_{r'} \circ \varrho_r)(m \otimes n).
\end{aligned}
\]
Therefore, $\varrho: \op{R} \lra E(M \otimes_T N)$ is a morphism of trusses and hence $M \otimes_T N$ is a right $R$-module. If $R$ admits identity $1$ and $N$ is a unital $R$-module, then $\varrho(1) = \id_{M \otimes N}$ and hence $\varrho$ (and $M \otimes_T N$) is also unital.

The other case is proven in a symmetric way.

\item[\ref{tens:item2}]  Similarly to the proof of statement~\ref{tens:item1}, one considers the assignment
\[
\varphi' : M \times N \lra M'\otimes_T N, \qquad (m,n) \lto \varphi(m) \otimes n.
\]
Since $\varphi$ is a morphism of right $T$-modules, $\varphi'$ is a $T$-balanced bilinear map, and hence it factors uniquely through
\[
\varphi \otimes N : M \otimes_T N \lra M'\ot_T N, \qquad m\otimes n \lto \varphi(m) \otimes n.
\]
Since $\varphi\ot N$ acts trivially on the elements in $N$, and the $R$-actions on $M\ot_TN$ and $M'\ot_T N$ are defined using the $R$-action on $N$ only, the resulting map is a homomorphism of right $R$-modules. 
The other case is proven in a symmetric way.

\item[\ref{tens:item3}] This follows immediately from assertions \ref{tens:item1}  and \ref{tens:item2}. \hfill \qedhere
\end{description}
\end{proof}
 
 \begin{proposition}\label{prop.adj}
Let $T$, $S$ be trusses and let $M$ be a $T$-$S$-bimodule. Then $-\ot_TM :\Mod\ds T \lra \Mod\ds S$ is the left adjoint functor to the  functor $\lhom T M-$.
\end{proposition}
\begin{proof}
The proof of this proposition follows  the same arguments as the proof of the corresponding statement for modules over rings. The only difference is that the distributivity of the tensor product over the heap ternary operation (rather than over a binary addition) should be employed whenever necessary (for example in showing that the unit and counit of the adjunction are morphisms of heaps). We only mention that the unit and the counit of the adjunction are explicitly given by
$$
\eta_X: X_T\lra \lhom S{M_S}{X\otimes_{T}M_{S}}, \qquad x\lto [m\lto x\ot m],
$$
$$
\eps_Y: \lhom S{M_S}{Y_S}\ot_T{M}_S\lto Y_S, \qquad \big[f_i\ot m_i\big]_{i=1}^{2n+1} \lto \big[f_i(m_i)\big]_{i=1}^{2n+1},
$$
for all right $T$-modules $X$ and  right $S$-modules $Y$.
\end{proof}

\begin{corollary}\label{cor:ass}
Let $R,S,T,U$ be trusses and let $A$ be an $R$-$S$-bimodule, $B$ be an $S$-$T$-bimodule and $C$ be a $T$-$U$-bimodule.  Then the map,
$$
\begin{aligned}
\alpha_{A,B,C}:  (A\otimes_R B)\otimes_T C& \longrightarrow A\otimes_R (B\otimes_T C),\\
(a\otimes b)\otimes c&\longmapsto a\otimes (b\otimes c).
\end{aligned}
$$
is an isomorphism of $R$-$U$-bimodules.
\end{corollary}
\begin{proof}
The assertion follows from Proposition~\ref{prop.adj} by standard arguments. 
\end{proof}

In view of the associativity of tensor products stemming from Corollary~\ref{cor:ass} we no longer need to write brackets in-between multiple tensor products.

The distributive laws for a truss $T$ mean that the multiplication map $\mu: T\times T\lra T$, $(s,t)\lto st$  is bilinear. Hence, there is a unique heap homomorphism $\hat{\mu}: T\ot T\lra T$. The associative law for $\mu$ is then reflected by the commutativity of the following diagram:
\begin{equation}\label{assoc.truss.tens}
\begin{gathered}
\xymatrix @R=20pt{T\ot T\ot T\ar[rr]^-{\hat{\mu}\ot T} \ar[d]_-{T\ot \hat{\mu}} && T\ot T\ar[d]^{\hat{\mu}}\\
T\ot T\ar[rr]^{\hat{\mu}} && T.}
\end{gathered}
\end{equation}
The existence of a map $\hat{\mu}$ satisfying \eqref{assoc.truss.tens} can be taken as the definition of the truss, provided that one carefully explains the meaning of $\tens$ (for example, by resorting to relations \eqref{simp.l} and \eqref{simp.r}) without referring to trusses, in order to avoid the {\em ignotum per ignotius} trap.

Similarly, if $T$ is a truss and $M$ a left $T$-module with action $\lambda_M$, then conditions \eqref{module2} and \eqref{module3} mean that $\lambda_M : T\times M\lra M$ is a bilinear map, so it induces a unique map $\hat{\lambda}_M : T\ot M\lra M$. Thus, exactly as in the case of modules over rings, a left module over a truss $T$ can be equivalently defined as an abelian heap $M$ together with a heap homomorphism  $\hat{\lambda}_M : T\ot M\lra M$ such that the diagram
$$
\xymatrix @R=20pt{T\ot T\ot M\ar[rr]^-{\hat{\mu}\ot T} \ar[d]_-{T\ot \hat{\lambda}_M} && T\ot M\ar[d]^{\hat{\lambda}_M}\\
T\ot M\ar[rr]^{\hat{\lambda}_M} && M}
$$
commutes, where $\hat{\mu}$ is the multiplication in $T$.  In a similar way, a right $T$-module can be equivalently described as a heap $M$ together with an associative right action $\hat{\varrho}_M: M\otimes T\lra M$. Taking these equivalent definitions of modules into account, one can interpret the tensor product as a coequalizer.

\begin{proposition}\label{prop.tensor.coeq}
Let $T$ be a truss. For a  right $T$-module $M$ and left $T$-module $N$, the tensor product $M\ot_T N$ is the coequalizer of  the following diagram of abelian heaps
\begin{equation}\label{eq:tenscoeq}
\xymatrix @C=40pt{M\ot T\ot N\ar@<.6ex>[rr]^{\hat{\varrho}_M\ot N}\ar@<-.6ex>[rr]_{M\ot \hat{\lambda}_N} && M\ot N,}
\end{equation}
where $\hat{\varrho}_M$ and $\hat{\lambda}_N$ are the corresponding actions.
\end{proposition}
\begin{proof}
Consider the structural morphisms $\phi: M\times N \lra M \ot N$ and $\varphi: M \times N \lra M \ot_T N$, part of the tensor product data. By definition, $\varphi$ is a bilinear map and so it factors uniquely through the morphism of abelian heaps
\[
\tilde{\varphi} : M \ot N \lra M \ot_T N,
\]
such that $\tilde{\varphi} \circ \phi = \varphi$. In addition, $\tilde{\varphi}$ satisfies
\[
\begin{aligned}
\tilde{\varphi}\left(\left(\hat{\varrho}_M \ot N\right)\left(m \ot t \ot n\right)\right) &= \tilde{\varphi}\left(m \cdot t \ot n\right) = \tilde{\varphi}\left(\phi\left(m \cdot t, n\right)\right) = \varphi(m \cdot t,n) \\
&= \varphi(m,t\cdot n) = \tilde{\varphi}\left(\left(M \ot \hat{\lambda}_N\right)\left(m \ot t \ot n\right)\right),
\end{aligned}
\]
for all $m \in M$, $n \in N$, $t\in T$, because $\varphi$ is $T$-balanced. Since every morphism involved is a morphism of abelian heaps, we conclude that $\tilde{\varphi}$ coequalizes the pair \eqref{eq:tenscoeq}. Now, let $(Q,q: M\ot N \to Q)$ be a pair coequalizing \eqref{eq:tenscoeq} as well. The composition $q \circ \phi$ is bilinear because
\[
\begin{aligned}
(q \circ \phi)\left(\left[m,m',m''\right],n\right) &= q\left(\left[m,m',m''\right]\ot n\right) \stackrel{\eqref{simp.l}}{=} q\left(\left[m\ot n,m'\ot n,m''\ot n\right]\right) \\ 
 &= \left[q\left(m\ot n\right),q\left(m'\ot n\right),q\left(m''\ot n\right)\right]\\
 & = \left[(q \circ \phi)\left(m, n\right),(q \circ \phi)\left(m', n\right),(q \circ \phi)\left(m'', n\right)\right] ,
\end{aligned}
\]
for all $m,m',m'' \in M$, $n \in N$, and analogously on the other side. Furthermore, it is also $T$-balanced because $q$ coequalizes \eqref{eq:tenscoeq}, and hence
\[
\begin{aligned}
(q \circ \phi)(m \cdot t,n) &= q(m\cdot t \ot n) = q\left(\left(\hat{\varrho}_M \ot N\right)\left(m \ot t \ot n\right)\right) \\
&= q\left(\left(M \ot \hat{\lambda}_N\right)\left(m \ot t \ot n\right)\right) = (q \circ \phi)(m,t \cdot n),
\end{aligned}
\]
for all $m \in M$, $n \in N$, $t \in T$. Thus, there exists a unique morphism of abelian heaps $\tilde{q} : M \ot _T N \to Q$ such that $\tilde{q}\circ \varphi = q \circ \phi$. In particular, $\tilde{q} \circ \tilde{\varphi} \circ \phi = q \circ \phi$ (by definition of $\tilde{\varphi}$) and since both $q$ and $\tilde{q} \circ \tilde{\varphi}$ are heap homomorphisms, the uniqueness part of the universal property of the tensor product entails that $\tilde{q} \circ \tilde{\varphi} = q$. Summing up, the pair $\left(M \ot _T N, \tilde{\varphi}\right)$ is the coequalizer of \eqref{eq:tenscoeq} in $\ahrd$.
\end{proof}

Corollary~\ref{cor:ass} and the above discussion can be formalized as the following result.
\begin{proposition}\label{prop:monoidal}
{~}
\begin{enumerate}[label=(\arabic*),ref=\emph{(\arabic*)},leftmargin=0.8cm]
\item\label{monoidal:1} The category $\ahrd$ with the tensor product of heaps as the operation and the singleton heap $\star$ as the unit object is a closed monoidal category. Unitors are projections, with inverses given by tensoring by $*$,
$$
A\otimes \star \lra A,\quad a\ot *\lto a, \qquad \star\otimes A \lra A,\quad  *\ot a\lto a.
$$
\item\label{monoidal:2}  For any truss $T$, the category $T\ds\Mod\ds T$ with the tensor product of $T$-modules and the truss $\uT$ as unit object is a closed monoidal category. Unitors are the actions with inverses given by insertion of identity as in Proposition~\ref{prop:tensunital}
\[
\begin{gathered}
M\ot_T\uT\lra M,  \quad m\otimes z\lto m\cdot z, \qquad \uT\ot_T M\lra M, \quad z\otimes m\lto z\cdot m,\\
M\lra M\ot_T \uT, \quad m\lto m\ot *, 					 \qquad M\lra \uT\ot_TM,  \quad m\lto *\ot m.
\end{gathered}
\]
In particular, for any unital truss $T$ the category $T_{1}\ds\Mod\ds T_{1}$ with the tensor product of $T$-modules and the truss $T$ as a unit object is a closed monoidal category. 
\item\label{monoidal:3}  A truss is a semigroup in the category $\ahrd$ and a unital truss is a monoid in $\ahrd$.  Conversely, any monoid in $\ahrd$ is a unital truss and any semigroup in $\ahrd$ is a truss.
\end{enumerate}
\end{proposition}
\begin{proof}
Proofs of \ref{monoidal:1} and \ref{monoidal:2} are analogous to the case of rings. 
Statement \ref{monoidal:3} follows by the discussion preceding the previous proposition, supplemented by the observation that if $T$ is a unital truss, then the unit of the corresponding monoid in $\ahrd$ is given by the map $\eta: \star \lra T$, $*\lto 1$, picking the identity.
\end{proof}

\begin{remark}\label{rem.cos}
In light of \cite[Proposition 9.1.6 and Theorem 9.3.8]{Ber:inv} the category $\ahrd$ of abelian heaps is a complete and cocomplete category. In light of Proposition~\ref{prop:monoidal}\ref{monoidal:1}, $(\ahrd,\otimes,\star)$ is a monoidal category. It can be easily checked that the switch map $\sigma: H\otimes H'\lra H'\otimes H, \, h\otimes h'\lto h'\otimes h,$ is well-defined and makes of $\ahrd$ a symmetric monoidal category. Finally, either because $\ahrd = \pMod{\star_1}$ and $\otimes=\otimes_\star$ or because 
\[
\xymatrix @R=0pt{
{\ahrd}\left({H\otimes H'},{H''}\right) \ar@{<->}[r] & {\ahrd}\left({H},{{\ahrd}\left({H'},{H''}\right)}\right), \\
f \ar@{|->}[r] & \left[h\lto \left[h'\lto f(h\otimes h')\right]\right], \\
\left[h\otimes h'\lto g(h)(h')\right] & g ,\ar@{|->}[l] 
}
\]
is a well-defined bijection, $\ahrd$ is a complete and cocomplete closed symmetric monoidal category, whence a \emph{cosmos} in the sense of J. B\'enabou (as reported in \cite[Introduction]{Street}).

 Let $T$ be a truss. It follows from what we observed in \S\ref{sec.truss} and from the fact that the composition $\lhom{T}{N}{P} \times \lhom{T}{M}{N} \to \lhom{T}{M}{P}$ is bilinear that $\lmod{T}$ is an $\left(\ahrd,\otimes,\star\right)$-enriched category.
\end{remark}

\begin{remark}
Since a truss is a semigroup object in a symmetric monoidal category with braiding given by the switch map, we can define the tensor product of trusses in the standard way. That is, if $S$ and $T$ are trusses with multiplications $\hat{\mu}_S$ and $\hat{\mu}_T$, respectively, then $S\otimes T$ is a truss with the multiplication:
$$
\hat{\mu}_{S\ot T}: \xymatrix{S\otimes T\otimes S\otimes T\ar[rr]^-{S\ot \sigma_{T,S}\ot T}&& S\otimes S\otimes T\otimes T\ar[rr]^-{\hat{\mu}_S\ot \hat{\mu}_S} && S\ot T.}
$$
If $S$ and $T$ are unital then also $S\ot T$ is a unital truss. Thus, similarly to the case of modules over rings, any $S$-$T$-bimodule $M$ can be understood  as a left $S\ot T^\circ$-module, where $T^\circ$ is the truss opposite to $T$. 
\end{remark}

We know from \cite[Proposition 9.1.6 and Theorem 9.3.8]{Ber:inv} that the category of modules over a truss is complete and cocomplete. We conclude this subsection with an independent argument that allows us to draw the same conclusion and, at the same time, that shows us a way to explicitly compute them, provided we know what they look like in $\ahrd$. As a by-product we provide an abstract description of free modules over a non-unital truss.

Let $T$ be a truss and $\uT$ be its unital extension as in \S\ref{ssec.unital}. By \cite[Chapter VII, Section 4]{MacLane}, the functor $\uT \otimes -: \ahrd \lra \lmod{(\uT)_1}$ is left adjoint to the forgetful functor $U':\lmod{(\uT)_1} \lra \ahrd$ (and hence it is called the \emph{free unital $\uT$-module functor}). Since the functor $\Uu:\lmod{T} \lra \lmod{(\uT)_1}$ of Theorem \ref{thm:unital} is the inverse of the restriction of scalars $\iota_T^*: \lmod{(\uT)_1} \lra \lmod{T}$ and since clearly $U'\circ \Uu$ coincides with the forgetful functor $U:\lmod{T} \lra \ahrd$,  the composition $\iota_T^*\circ (\uT \ot -)$ is left adjoint to $U$. Once observed that $\iota_T^*\circ (\uT \ot -)$ is naturally isomorphic to tensoring by the left $T$-module $\iota_T^*(\uT)$, we conclude that the functor $\iota_T^*(\uT) \ot - : \ahrd \lra \lmod{T}$ is left adjoint to the forgetful functor $U$ and hence it is called the \emph{free $T$-module functor}. 

As a consequence of the existence of the adjunction $(\iota_T^*(\uT )\otimes -)  \dashv U$, we can consider the monad $\TT \coloneqq U(\iota_T^*(\uT )\otimes -) = U(\iota_T^*(\uT)) \ot - : \ahrd \lra \ahrd$. In view of the fact that $\uT$ is a monoid in $\ahrd$, the Eilenberg-Moore category of algebras for the monad $U'(\uT\ot -)$ on $\ahrd$ is exactly the category of unital $\uT$-modules. Since $U'(\uT\ot -) = U'\Uu \iota_T^*(\uT\ot -) = \TT$, the Eilenberg-Moore category $\mathsf{EM}^\TT$ is exactly the category of modules over $T$.

\begin{proposition}\label{prop:colimits}
For a truss $T$, the forgetful functor $U: \lmod{T} \lra \ahrd$ creates and preserves all limits and all colimits that exist in $\ahrd$. That is to say, if a functor $D : \cC \to \lmod{T}$ is such that $U\circ D$ has a (co)limit $H$ in $\ahrd$, then $D$ has a (co)limit $\widehat{H}$ in $\lmod{T}$ and $U(\widehat{H}) = H$. In particular, $\lmod{T}$ is complete and cocomplete.
\end{proposition}

\begin{proof}
Since $U(\iota_T^*(\uT)\otimes -)$ coincides with $U(\iota_T^*(\uT))\otimes - :\ahrd\lra\ahrd$ and since it is left adjoint to ${\ahrd}(U(\iota_T^*(\uT)),-)$, it preserves all colimits. Therefore, the statement follows from \cite[Propositions 4.3.1 and 4.3.2]{Borceaux} and (co)completeness of $\ahrd$.
\end{proof}

Proposition \ref{prop:colimits} amounts to say that (co)limits of $T$-modules can be obtained (up to isomorphism) by endowing the corresponding (co)limits of abelian heaps with a suitable $T$-action, as it happens with the product, for instance.

\begin{proposition}\label{prop:free}
Let $X$ be a non-empty set and $T$ be a unital truss. Denote by $\Aa(X)$ the free abelian heap over $X$ and by $\Tt^X$ the free unital $T$-module over $X$. Then $\Tt^X\cong T \otimes \Aa(X)$ as $T$-modules. In particular, the following diagram of functors commutes
\[
\xymatrix @C=30pt{
\lmod{T_1} \ar@<+0.5ex>[dr]^{U'} \ar@<+0.5ex>[dd]^{\Ff'} & \\
 & \ahrd\ . \ar@<+0.5ex>[dl]^{\Ff''} \ar@<+0.5ex>[ul]^-{T \otimes-} \\
\Set \ar@<+0.5ex>[uu]^{\Tt} \ar@<+0.5ex>[ur]^{\Aa} & 
}
\]
\end{proposition}

\begin{proof}
Fix $e\in X$. In view of \cite[Remark 3.5 and Proposition 3.9]{BrzRyb:mod} we know that 
\[
\Aa(X)\cong\hH\left(\bigoplus_{x\in X\setminus\{e\}}\ZZ\right) \cong \boks{x\in X}{}{\Aa(\{x\})},
\]
and the isomorphism $\Aa(X)\cong \boks{x\in X}{}{\Aa(\{x\})}$ is independent from the choice of $e\in X$. Now, since $T\otimes -$ is cocontinuous (because it is the left adjoint functor of the forgetful functor), we have the following chain of isomorphisms of left $T$-modules
\[
T \otimes \Aa(X) \cong T \otimes \left(\boks{x\in X}{}{\Aa(\{x\})}\right) \cong \boks{x\in X}{}{\left(T \otimes \Aa(\{x\})\right)}.
\]
Consider $\Aa(\{x\})$. As a set, $\Aa(\{x\})=\{x\}$ with the ternary operation $[x,x,x]=x$. This makes it clear that $\star \lra \Aa(\{x\}), *\lto x$, is an isomorphism of (abelian) heaps. Therefore, $T \otimes \Aa(X) \cong \boks{x\in X}{}Tx \cong \Tt^X$.
\end{proof}

Let us make explicit the foregoing isomorphism in an extremely easy example.

\begin{example}
Let $X = \{a,b\}$ be a set with two elements. The free abelian heap $\Aa(X)$ on $X$ can be realized as the set
\[\{a,b,aba,bab,ababa,babab,abababa,bababab,\ldots\}\]
with bracket given by concatenation and (symmetric) pruning. Then, for instance,
\[
t \ot ababa \longleftrightarrow (ta)(tb)(ta)(tb)(ta) = ([t,1,t,1,t]a)([t,1,t]b)(1a)(1b)(1a).
\]
\end{example}

\begin{corollary}[of Proposition \ref{prop:free}]
Let $T$ be a truss and $\uT$ its unital extension. Denote by $\Ff : \lmod{T} \lra \Set$ the forgetful functor. In the following diagram of adjunctions, the subdiagram involving only the right adjoints is commutative
\[
\xymatrix @!0 @C=50pt @R=35pt{
 & \lmod{(\uT)_1} \ar@<+0.4ex>[dl]^(0.4){\iota_T^*} \ar@<+0.4ex>[ddrr]^-{U'} & & \\
\lmod{T} \ar@<+0.4ex>[ur]^(0.4){\Uu} \ar@<+0.4ex>[ddrr]^(0.7){\Ff} \ar@<+0.4ex>[drrr]^(0.45){U} & & & \\
 & & & \ahrd\ . \ar@<+0.4ex>[uull]^(0.7){\uT\ot -} \ar@<+0.4ex>[ulll]^-{\uT\ot -} \ar@<+0.4ex>[dl]^-{\Ff''} \\
 & & \Set \ar@<+0.4ex>[ur]^-{\Aa} \ar@<+0.4ex>@{.>}[uull]^-{\uT\ot \Aa(-)} &
}
\]
In particular, the free $T$-module over a set $X$ is $\uT\ot\Aa(X)$.
\end{corollary}

Concretely, when $T$ is a not necessarily unital truss we can describe the free $T$-module over a set $X$ as the direct sum of abelian heaps
\[
\Tt^X \coloneqq \boks{x \in X}{}{\uT x},
\]
with the $T$-action given component-wise, that is,
\[
t \cdot \left[z_1x_1,\ldots,z_{2k+1}x_{2k+1}\right] = \left[\left(\iota_T(t)z_1\right)x_1,\ldots,\left(\iota_T(t)z_{2k+1}\right)x_{2k+1}\right],
\]
for all $x_1,\ldots,x_{2k+1} \in X$, 
$z_1,\ldots,z_{2k+1} \in \uT$ and $t \in T$. The canonical map $\iota_X : X \lra \Tt^X$ (that is, the unit of the adjunction $\uT \ot \Aa(-)\ \dashv\ \Ff$) sends every $x \in X$ to $*x \in \uT x$. The other way around, the counit $\epsilon$ of the adjunction $\uT \ot \Aa(-)\ \dashv\ \Ff$ realizes every $T$-module $M$ as a quotient of a free one:
\[
\Tt^{\Ff(M)} \cong \uT\ot\Aa\left(\Ff(M)\right) \xrightarrow{\epsilon_M} M
\]
(since $\Ff$ is faithful, every component of $\epsilon$ is an epimorphism in view of \cite[Theorem IV.3.1]{MacLane})

\section[Morita theory and the Eilenberg-Watts theorem]{\for{toc}{Morita theory and the Eilenberg-Watts theorem}\except{toc}{Adjoint functors between categories of modules over trusses: Morita theory and the Eilenberg-Watts theorem}}\label{sec.morita}

Given two trusses $S,T$ and a $T\mbox{-}S$-bimodule $M$ we already know that the functor $M\otimes_S - : \pMod{S} \lra \pMod{T}$ is left adjoint to the functor $\lhom{T}{M}{-}:\pMod{T} \lra \pMod{S}$. Our aim in the present section is to show that, if $T$ and $S$ are (unital) trusses, then any heap functor $L:\pMod{S} \lra \pMod{T}$ which admits a right adjoint is of the form $P\otimes_S -$ for a suitable (unital) $T\mbox{-}S$-bimodule $P$. We will conclude the section by proving a heap analogue of the celebrated Eilenberg-Watts Theorem, giving an intrinsic characterization of left adjoint functors. Recall that a functor $F:\pMod{S} \lra \pMod{T}$ is a heap functor provided that, for all $M,N\in \lmod{S}$, the functions $F_{M,N}$ defined by equation  \eqref{heap.fun} are morphisms of heaps.
 Recall also that the unital extension $\uT$ of a truss $T$ is a $T$-$T$-bimodule via the truss homomorphism $\iota_T : T \lra \uT$.

\begin{lemma}\label{lem:heapbimod}
Let $S,T$ be trusses and let $F:\pMod{S} \lra \pMod{T}$ be a heap functor between their categories of modules.  Then $P \coloneqq  F(\iota_S^*(\uS))$ is a $T\mbox{-}S$-bimodule. Furthermore, if $S$ is unital and $F:\pMod{S_1} \lra \pMod{T}$ is a heap functor, then $P' \coloneqq F(S)$ is a $T\mbox{-}S$-bimodule which is unital  as right $S$-module. 
\end{lemma}

\begin{proof}
 To simplify notation we write $\uS$ instead of $\iota_S^*(\uS)$. For every $s\in S$, consider the left $S$-module morphism
\[\rho_s : \uS\lra \uS, \qquad  z\lto z \cdot \iota_S(s).\]
Clearly, $\rho_{ss'} = \rho_{s'}\circ \rho_s$  and, by the right distributive law of the action of $S$ on $\uS$, $\rho_{[s,s',s'']} = \left[\rho_s,\rho_{s'},\rho_{s''}\right]$ in $\lhom{S}{\uS}{\uS}$ for all $s,s',s''\in S$. Therefore the map 
$$
\rho: S^\circ \lra E_S(\uS), \qquad s\lto \rho_s,
$$
is a homomorphism of trusses. Since $F$ is a heap functor, the composite
$$
\xymatrix{
S^\circ\ar[rr]^-{\rho} &&  E_S(\uS) \ar[rr]^-{F_{\uS,\uS}} && E_T(F(\uS)) =E_T(P),}
$$
where $F_{\uS,\uS}$ is defined by \eqref{heap.fun}, is a morphism of trusses. As a consequence, $P$ inherits the structure of a $T\mbox{-}S$-bimodule. 
If $S$ is unital, we may perform the same construction using $S$ instead of $\uS$ and $P'=F(S)$ becomes unital as a right $S$-module.
\end{proof}

\begin{proposition}\label{prop:LA}
 Let $S,T$ be trusses. A heap functor $L:\pMod{S} \lra \pMod{T}$ admits a right adjoint if and only if it is naturally equivalent to $P\otimes_S -$ for a suitable $T$-$S$-bimodule $P$. Namely, $P \coloneqq L(\iota_S^*(\uS))$. If, in addition, $S$ is unital then a heap functor $L:\pMod{S_1} \lra \pMod{T}$ admits a right adjoint if and only if it is naturally equivalent to $P'\otimes_S -$ for a suitable $T\mbox{-}S$-bimodule $P'$, unital as a right $S$-module. Namely, $P'\coloneqq L(S)$.
\end{proposition}

\begin{proof}
We already know from Lemma \ref{lem:heapbimod} that $P \coloneqq L(\iota_S^*(\uS))$ is a $T\mbox{-}S$-bimodule. Let us denote by $R:\pMod{T}\lra \pMod{S}$ the right adjoint to $L$ and let us consider the adjunction isomorphism
\[
\Phi_{\uS,N}:\lhom{T}{P}{N} = \lhom{T}{L(\iota_S^*(\uS))}{N} \cong \lhom{S}{\iota_S^*(\uS)}{R(N)}
\]
for all $N$ in $\pMod{T}$. Then, for all $s\in S$ and $f\in \lhom{T}{P}{N}$, \[
\begin{aligned}
\Phi_{\uS,N}(s\cdot f) & = \Phi_{\uS,N}( f \circ L(\rho_s)) = \left(\Phi_{\uS,N} \circ \lhom{T}{L(\rho_s)}{N}\right)(f) \\
 & = \left(\lhom{S}{\rho_s}{R(N)}\circ \Phi_{\uS,N}\right)(f) \\
 & = \Phi_{\uS,N}(f)\circ \rho_s = s\cdot \Phi_{\uS,N}(f),
\end{aligned}
\]
that is $\Phi_{\uS,N}$ is a left $S$-linear isomorphism natural in $N\in\pMod{T}$. Since $\iota_S^*$ is the inverse of $\Uu$, we have further
\[
\lhom{S}{\iota_S^*(\uS)}{R(N)} \cong \iota_S^*\left(\lhom{\uS}{\uS}{\Uu(R(N))}\right)
\]
as left $S$-modules. Now, in view of the fact that both $\uS$ and $\Uu(R(N))$ are unital, the assignment 
\[
\lhom{\uS}{\uS}{\Uu(R(N))} \lra \Uu(R(N)), \qquad f\lto f\left(1_{\uS}\right),
\]
is an isomorphism of heaps, natural in $N$, which is also left $\uS$-linear. Therefore, 
\[\iota_S^*\left(\lhom{\uS}{\uS}{\Uu(R(N))}\right) \cong \iota_S^*(\Uu(R(N))) \cong R(N)\]
and we conclude that $R\cong \lhom{T}{P}{-}$ as functors from $\pMod{T}$ to $\pMod{S}$. Being the left adjoint functor to $\lhom{T}{P}{-}$, $L\cong P \otimes_S-$ as desired, by the uniqueness of adjoints up to isomorphism. Finally, in case $S$ is unital one may mimic the same procedure starting with $P'=L(S)$ instead.
\end{proof}

With Proposition~\ref{prop:LA} we showed that any functor between module categories over trusses which admits a right adjoint is naturally obtained by taking tensor products with suitable bimodules. Now we prove an analogue of the Eilenberg-Watts theorem for modules over trusses which, in turn, allows us to give an intrinsic characterisation of when a functor is given by tensoring by a bimodule (and hence it is a left adjoint) in terms of properties of the functor itself.

\begin{theorem}[Eilenberg-Watts Theorem for trusses]\label{thm:EW}
 Let $T$ and $S$ be trusses. If $F:\pMod{T} \lra \pMod{S}$ is a cocontinuous heap functor, then
$$
F(-)\cong P\otimes_{T}-,
$$
for  an $(S,T)$-bimodule $P.$ Namely, $P \coloneqq F(\iota_T^*(\uT))$. If, in addition, $T$ is unital and $F:\pMod{T_1} \lra \pMod{S}$ is a cocontinuous heap functor, then
$$
F(-)\cong P'\otimes_{T}-,
$$
for an $(S,T)$-bimodule $P',$ unital as right $T$-module. Namely, $P' \coloneqq F(T)$.
\end{theorem}

\begin{proof}
 We prove only the first claim and, for the sake of simplicity, we write $\uT$ instead of $\iota_T^*(\uT)$. Let $X$ be a $T$-module. One can consider a coequalizer diagram
\begin{equation}\label{coeq2}
\xymatrix @C=40pt {\Ker(\pi)\ar@<.6ex>[r]^{p_1}\ar@<-.6ex>[r]_{p_2}& \Tt^{X}\ar[r]^{\pi} & X},
\end{equation}
as in the proof of Proposition~\ref{prop:epicoeq}, where $\Tt^{X}$ is the free $T$-module over the set underlying $X$, $\pi$ is the canonical epimorphism, $\Ker(\pi)=\{(x,y)\in \Tt^{X}\times \Tt^{X} \mid \pi(x)=\pi(y)\}$ with the component-wise $T$-module structure, and $p_{1},p_{2}$ are the (restrictions of the) two canonical projections. One can extend diagram $\eqref{coeq2}$ to

\begin{equation}\label{coeq3}
\xymatrix{\Tt^{\Ker(\pi)}\ar@<.6ex>[rr]^{p_1'}\ar@<-.6ex>[rr]_{p_2'}&& \Tt^{X}\ar[rr]^{\pi} && X },
\end{equation}
where $\pi':\Tt^{\Ker(\pi)}\lra \Ker(\pi)$, $p_{1}'=p_{1}\circ \pi'$ and $p_{2}'=p_{2}\circ \pi'$. Since $\pi'$ is an epimorphism, \eqref{coeq3} is a coequalizer diagram as well.
By Lemma~\ref{lem:heapbimod}, $P\coloneqq F(\uT)$ inherits the structure of an $(S,T)$-bimodule from the fact that  $F$ is a heap functor. Since $F$ is a cocontinuous functor and in view of Proposition \ref{prop:tensunital}, there is the following chain of natural isomorphisms: 
$$
F\left(\Tt^{X}\right) = F\left(\boks{x\in X}{}{\uT}\right) \cong \boks{x\in X}{}{F(\uT)}\cong\boks{x\in X}{}{(P\otimes_{T} \uT)}\cong P\otimes_{T}\boks{x\in X}{}{\uT}= P\otimes_{T} \Tt^{X}.
$$
Moreover we can fill in a diagram
$$
\begin{gathered}
\xymatrix @R=20pt{F(\Tt^{\ker(\pi)})\ar[d]^{\cong}\ar@<.6ex>[rr]^{F(p_1')}\ar@<-.6ex>[rr]_{F(p_2')}&&F(\Tt^{X})\ar[d]^{\cong}\ar[rr]^{F(\pi)} && F(X)\ar@{.>}[d]^{\psi}\\
P\otimes_{T} \Tt^{\ker(\pi)}\ar@<.6ex>[rr]^{P\otimes_T p_1'}\ar@<-.6ex>[rr]_{P\otimes_T p_2'}&&P\otimes_{T}  \Tt^{X}\ar[rr]^{P\otimes_T \pi} &&P\otimes_{T}  X, }
\end{gathered}
$$
where both horizontal diagrams are coequalizers obtained from $\eqref{coeq3}$, because $F$ and $P \otimes_T -$ preserve colimits, and $\psi$ is the isomorphism induced by their universal property. It can be checked, by resorting to the uniqueness of the morphisms induced at the level of the coequalizers, that $\psi$ is in fact natural in $X$.
\end{proof}

\begin{corollary}
Let $T,S$ be trusses. A functor $F:\pMod{T} \lra \pMod{S}$ is a left adjoint if and only if it is a cocontinuous heap functor. If, in addition, $T$ is unital then $F:\pMod{T_1} \lra \pMod{S}$ is a left adjoint functor if and only if it is a cocontinuous heap functor.
\end{corollary}

\begin{proof}
The statements follow from Proposition \ref{prop:LA}, Theorem \ref{thm:EW} and the fact that $P \otimes_T -$ is cocontinuous, heap and a left adjoint functor.
\end{proof}

 Assume that $S$ and $T$ are unital trusses. A key question related to the Morita theory for trusses is what can be said when $\pMod{T_1}\cong \pMod{S_1}$.  Notice that this covers the non-unital case as well, since in that case $\lmod{T}\cong \lmod{S}$ if and only if $\lmod{(\uT)_1} \cong \lmod{(\uS)_1}$. 

\begin{theorem}\label{thm:Morita}
Let $T,S$ be unital trusses. The following statement are equivalent:
\begin{enumerate}[label=(\arabic*), ref=\emph{(\arabic*)}, leftmargin=0.8cm]
\item\label{item:Morita1} $\pMod{T_1}\cong \pMod{S_1}$.
\item\label{item:Morita2} There exist unital bimodules ${_SP_T}$ and ${_TQ_S}$ together with an $S$-bilinear isomorphism $\ev: P \otimes_T Q \lra S$ and a $T$-bilinear isomorphism $\db: T \lra Q\otimes_S P$ such that
\begin{equation}\label{eq:zigzag}
(Q\otimes_S \ev)\circ(\db \otimes_T Q) = \id_Q \quad \text{ and } \quad (\ev \otimes _S P)\circ(P \otimes_T \db)=\id_P.
\end{equation}
\item\label{item:Morita3} There exist unital bimodules ${_SP_T}$ and ${_TQ_S}$ together with an $S$-bilinear isomorphism $\db': S \lra P \otimes_T Q$ and a $T$-bilinear isomorphism $\ev': Q\otimes_S P \lra T$ such that
\begin{equation*}
(P\otimes_T \ev')\circ(\db' \otimes_S P) = \id_P \quad \text{ and } \quad (\ev' \otimes _T Q)\circ(Q \otimes_S \db')=\id_Q.
\end{equation*}
\end{enumerate}
\end{theorem}

\begin{proof}
Since the proofs of \ref{item:Morita1}$\iff$\ref{item:Morita2} and of \ref{item:Morita1}$\iff$\ref{item:Morita3} are similar, we will present explicitly only the first one and leave the second one to the reader.

To show that \ref{item:Morita1} implies \ref{item:Morita2}, assume that $L:\pMod{T_1}\lra \pMod{S_1}$ and $R:\pMod{S_1}\lra \pMod{T_1}$ are inverse equivalences (or quasi-inverse functors). Equivalently, we may assume that $L$ is left adjoint to $R$ and that the counit $\eps:L\circ R \lra \id$ and the unit $\eta:\id \lra R\circ L$ of this adjunction are natural isomorphisms. In light of Proposition~\ref{prop:LA}, there exists a unital $(S,T)$-bimodule $P$ such that $L\cong P\otimes_T -$. At the same time, we may look at $R$ as left adjoint to $L$ with counit $\eta^{-1}:R\circ L \lra \id$ and unit $\eps^{-1}: \id \lra L\circ R$, and hence there exists a unital $T\mbox{-}S$-bimodule $Q$ such that $R \cong Q \otimes_S -$. Consider the following isomorphisms
\[
\begin{gathered}
\db \coloneqq  \left(\xymatrix{T \ar[r]^-{\eta_T} & R(P\otimes_T T) \ar[r]^{\cong} & Q\otimes_S P \otimes_T T \ar[r]^-{\cong} & Q\otimes _S P} \right), \\
\ev  \coloneqq  \left(\xymatrix{ P \otimes_T Q \ar[r]^-{\cong} & P\otimes_T Q\otimes_S S \ar[r]^-{\cong} & P\otimes_T R(S) \ar[r]^-{\eps_S} & S}\right).
\end{gathered}
\]
First, we are going to show that $\eta$ and $\eps$ can be written in terms of $\ev $ and $\db$. Then, we will see how the triangular identities for unit and counit reflect on $\ev $ and $\db$. For every left $T$-module $M$ and for every $m\in M$, consider the left $T$-module homomorphism $\rho_m: T \lra M$,  $t\lto t\cdot m$. By naturality of $\eta$,
\[
\eta_M(m) = \left(\eta_M \circ \rho_m\right)(1_T) = \left(Q\otimes _S P \otimes_T \rho_m\right)(\eta_T(1_T)) = \db(1_T)\otimes_T m.
\]
Similarly, for every left $S$-module $N$ and for every $n\in N$ we consider the left $S$-module homomorphism $\rho_n:S\lra N$, $s\lto s\cdot n$ and, by naturality of $\eps$,
\begin{equation}\label{eq:epsiev}
\begin{gathered}
\eps_N\left(p\otimes_T q\otimes_S n\right) = \left(\eps_N \circ \left( P \otimes_T Q\otimes_S \rho_n\right)\right)\left(p\otimes_T q\otimes_S 1_S\right) \\
 = \rho_n\left(\eps_S\left(p\otimes_T q\otimes_S 1_S\right)\right) = \ev (p\otimes_T q)\cdot n.
\end{gathered}
\end{equation}
Let us write explicitly $\db(1_T) = [q_i\otimes_S p_i]_i$ and $\ev (p\otimes_T q) = q(p)$. By the triangular identities, for every $S$-module $N$ and for all $q\in Q, n\in N$,
\[
\begin{aligned}
q\otimes_S n & = \left((Q\otimes_S \eps_N)\circ\eta_{Q\otimes_S N}\right)(q\otimes_S n) = (Q\otimes_S \eps_N)([q_i\otimes_S p_i]_i\otimes_T q\otimes_S n) \\
& = [q_i\cdot q(p_i)]_i\otimes_{S}n.
\end{aligned}
\]
In a  similar way, for every $T$-module $M$ and for all $p\in P$ and $m\in M$,
\[
\begin{aligned}
p\otimes_T m & = \left(\eps_{P\otimes_T M}\circ (P\otimes_T \eta_M)\right)(p\otimes_T m) = \eps_{P\otimes_T M}(p\otimes_T [q_i\otimes_S p_i]_i\otimes _S m ) \\
 & = [q_i(p)\cdot p_i]_i\otimes_T m.
\end{aligned}
\]
In particular, for $N=S$, $n=1_S$, $M=T$, $m=1_T$, we find that
\begin{equation}\label{eq:zigzag2}
[q_i\cdot q(p_i)]_i = q \qquad \text{and} \qquad [q_i(p)\cdot p_i]_i = p,
\end{equation}
for all $p\in P$ and $q\in Q$.
Concerning bilinearity, on the one hand, for every $t\in T$,
\[
\begin{gathered}
{[q_i\otimes_S p_i]_i \cdot t = [q_i\otimes_S p_i \cdot t]_i = [q_i\otimes_S [q_j(p_i \cdot t)\cdot p_j]_j]_{i} = [q_i\cdot q_j(p_i \cdot t) \otimes_S p_j]_{i,j}} \\
 = [[q_i\cdot q_j(p_i \cdot t)]_i \otimes_S p_j]_{j} \stackrel{(*)}{=} [t\cdot q_j \otimes_S p_j]_{j} = t\cdot [q_j \otimes_S p_j]_{j},
\end{gathered}
\]
where $(*)$ follows from the fact that $\ev $ is a $T$-balanced map. Whence $\db$ is a $T$-bimodule homomorphism. On the other hand,
\[
\ev (p \otimes_T q \cdot s) = \eps_S(p \otimes_T q \otimes_S \rho_s(1_S)) \stackrel{\eqref{eq:epsiev}}{=} \ev (p\otimes_T q)s,
\]
and hence $\ev $ is an $S$-bimodule homomorphism. In view of this, \eqref{eq:zigzag2} can now be rewritten as \eqref{eq:zigzag}.

Conversely, to prove that \ref{item:Morita2} implies \ref{item:Morita1} consider the functors $P\otimes_T - : \pMod{T} \lra \pMod{S}$ and $Q\otimes_S- :\pMod{S}\lra \pMod{T}$. If we define unit and counit by
\[
\begin{gathered}
\eta_M\coloneqq  \db\otimes_T M : M \lra Q\otimes_S P\otimes_T M, \\
\eps_N \coloneqq  \ev  \otimes_SN : P\otimes_T Q\otimes_S N \lra N,
\end{gathered}
\]
for every $T$-module $M$ and every $S$-module $N$, then the zigzag identities \eqref{eq:zigzag} entail that $P\otimes_T -$ is left adjoint to $Q\otimes_S -$ and the fact that $\eta$ and $\eps$ are natural isomorphisms implies in addition that these two functors define an equivalence of categories.
\end{proof}

\begin{remark}\label{rem:Pdual}
By checking closely the proof of Theorem \ref{thm:Morita}, one may notice that $R\cong \lhom{S}{P}{-}$ as the right adjoint functor of $L\cong P\otimes_T -$ and $R\cong Q\otimes _S -$ since it is a left adjoint functor itself. Therefore,
\[
{^*P} \coloneqq  \lhom{S}{{}_SP}{S} \cong R(S) \cong Q\otimes_S S \cong Q
\]
as $T\mbox{-}S$-bimodules. Analogously, $P\cong {Q^*} \coloneqq  \lhom{S}{Q_S}{S}$ as $(S,T)$-bimodules.
Moreover, we point out that any argument provided for left modules would hold symmetrically for right modules.
\end{remark}

A distinguished functor $F:\pMod{S} \lra \pMod{T}$ which plays an important role in the study of the Frobenius property for trusses is the \emph{restriction of scalars functor} $F=f^*$ associated with a truss homomorphism $f:T\lra S$. This is the faithful functor sending every left $S$-module $M$ to the left $T$-module ${_fM}\coloneqq f^*(M)$ having the same underlying heap structure but action given by $t\cdot m = f(t)\cdot m$ for all $t\in T, m\in M$, and sending every $S$-linear morphism to itself, but now seen as a $T$-linear map. We already saw examples of restriction of scalars functors in Theorem \ref{thm:unital} and Proposition \ref{prop:free} (the forgetful functor $U':\lmod{T_1} \lra \ahrd$ can be seen as the restriction of scalars along the unital truss homomorphism $\eta:\star \lra T$).

\begin{proposition}
The restriction of scalars functor $F:\pMod{S}\lra \pMod{T}$ associated with a truss homomorphism $f:T\lra S$ satisfies 
\[
\lhom{S}{{\left(\uS\right)_f}}{-} \ \cong \ F \ \cong \ {_f(\uS)}\otimes_S -.
\]
In particular, there is an  \emph{adjoint triple} of functors:  
$$
\left(\uS\right)_f\otimes_T - \;\dashv\; F \;\dashv\; \lhom{T}{{_f(\uS)}}{-}.
$$
\end{proposition}

\begin{proof}
For every left $S$-module $M$, consider the assignments
\[
\xymatrix @R=0pt{
\lhom{S}{\left(\uS\right)_f}{M} \ar@{<->}[r] & {_fM} \\
\phi \ar@{|->}[r] & \phi(*) \\
{\left[\begin{array}{c}*\lto m \\ s \lto s\cdot m\end{array}\right]} & m \ar@{|->}[l]
}
\qquad\qquad
\xymatrix @R=0pt{
{_fM} \ar@{<->}[r] & {_f(\uS)}\otimes_S M \\m \ar@{|->}[r] & {*}\otimes _S m \\
[z_i\cdot m_i]_{i=1}^{2k+1} & [z_i \otimes_S m_i]_{i=1}^{2k+1} \ar@{|->}[l]
}
\]
as in the proof of Proposition \ref{prop:tensunital}. They are $T$-linear isomorphisms, natural in $M$.
\end{proof}

\section{Small-projective modules and the Dual Basis Property}\label{sec:tiny}

Let $T$ be a truss. The conditions in Theorem \ref{thm:Morita} and the subsequent observations in Remark \ref{rem:Pdual} call for a closer analysis of $T$-modules admitting a dual basis $\db$ and evaluation $\ev$ morphisms.

\begin{definition}
A module $P$ over a truss $T$ is said to satisfy the \emph{dual basis property} (DBP for short) if there exist an odd integer $s = 2k+1$, an element $(e_1,\ldots,e_{s}) \in P^s$ and an element $(\phi_{1},\ldots,\phi_{s}) \in \lhom{T}{P}{T}^s$ such that, for all $p\in P$,
\begin{equation}\label{eq:dbp}
p=\LL \phi_{1}(p) \cdot e_{1},\ldots,\phi_{s}(p) \cdot e_{s}\RR.
\end{equation}
We call the pair $\left\{(e_1,\ldots,e_{s}),(\phi_{1},\ldots,\phi_{s})\right\}$ a \emph{dual basis} for $P$.
\end{definition}

\begin{example}\label{ex:dbp1}
The following examples are immediate from the definition.
\begin{enumerate}[leftmargin=0.8cm]
\item\label{item:dbp1} The empty $T$-module $\varnothing$ never satisfies the DBP. 
\item If $T$ is unital, then $P=T$ itself satisfies the DBP with $e_1 = 1_T$ and $\phi_1 =\id_T$.
\item The singleton $T$-module $\star$ satisfies the DPB if and only if $T$ admits a left absorber. Indeed, if $T$ admits a left absorber $o$ then $\star$ satisfies the DBP with $e_1 = *$ and $\phi_1: \star \lra T, * \lto o$. Conversely, if $\star$ satisfies the DBP then $\phi_1(*) \in T$ is a left absorber.
\item If $T$ is a unital truss with identity $1_T$ and $S$ is a truss with a left absorber $a$, then $T$ satisfies the DBP  as an $S \times T$-module with $e_1 = 1_T$ and $\phi_1 : T \lra S\times T, t \lto (a,t)$. For example, if we take $S = E(T)^\circ$ with $a : T \lra T, t \lto 1_T$, then $T$ satisfies the DBP as an $\left(E(T)^\circ \times T\right)$-module.
\end{enumerate}
\end{example}

 As usual, let $\uT$ be the unital extension of $T$. Set ${^*P}\coloneqq \lhom{T}{P}{\uT}$. It is a right $T$-module with $(f \cdot t)(p) \coloneqq f(p)t$ for all $f \in {^*P}$, $t \in T$ and $p\in P$.

\begin{remark}
\begin{enumerate}[label=(\alph*),ref=(\alph*),leftmargin=0.8cm]
\item If $P$ satisfies the DBP, then ${^*P}$ satisfies the DBP. For every $i=1,\ldots,s$, consider the right $T$-linear morphism 
\[
\ev_{i}: {^*P} \lra \uT, \qquad \alpha\lto \alpha\left(e_i\right).
\] 
Then, for all $\alpha\in {^*P}$,
\[
\alpha(p) = \alpha\left(\LL\phi_k(p) \cdot e_{k}\RR_{k=1}^s\right) = \LL\phi_k(p)\alpha\left(e_{k}\right)\RR_{k=1}^s = \LL\phi_k\cdot \ev_{k}\left(\alpha\right)\RR_{k=1}^s(p),
\]
for all $p\in P$, whence $\alpha = \LL\phi_k\cdot \ev_{k}\left(\alpha\right)\RR_{k=1}^s$.
\item If $P$ satisfies the DBP, then, for every $T$-module $M$ and for every $f:P\lra M$,
\[
f = \LL\phi_kf\left(e_{k}\right)\RR_{k=1}^s
\]
in $\lhom{T}{P}{M}$, where $\phi_kf\left(e_{k}\right) : P \lra M,\, p\lto \phi_k(p)f\left(e_{k}\right)$.
\end{enumerate}
\end{remark}

\begin{theorem}\label{thm:toostrong}
Let $T$ be a truss and $P$ a left $T$-module. The following properties are equivalent
\begin{enumerate}[label=\emph{(\arabic*)},ref=(\arabic*),leftmargin=0.8cm]
\item\label{item:fgp1} The functor $\lhom{T}{P}{-}:\pMod{T}\lra\ahrd$ is right exact (\ie it preserves finite colimits) and $P$ is finitely generated.
\item\label{item:fgp2} The module $P$ satisfies the DBP. 
\item\label{item:fgp3} There exist a $T$-bilinear morphism $\ev:P\otimes {^*P} \lra \uT$ and a morphism of abelian heaps $\db:\star \lra {^*P}\otimes_T P$ (\ie a $\star$-bilinear morphism) such that
\[
\left(\ev\otimes_T P\right)\circ \left(P\otimes \db\right) = \id_P \quad \text{and} \quad \left({^*P}\otimes_T \ev\right)\circ\left(\db\otimes {^*P}\right) = \id_{{^*P}},
\]
up to the canonical isomorphisms $P\otimes \star\cong P \cong \uT\otimes_TP$, ${^*P} \otimes_T \uT\cong {^*P}\cong \star \otimes {^*P}$.
\item\label{item:fgp4} The functor $\lhom{T}{P}{-}$ is naturally isomorphic to the functor ${^*P}\otimes_T -$.
\item\label{item:fgp5} The functor $\lhom{T}{P}{-}$ is cocontinuous (\ie it preserves small colimits).
\end{enumerate}
\end{theorem}

\begin{proof}
\ref{item:fgp1} $\Rightarrow$ \ref{item:fgp2}.  Assume that the functor $\lhom{T}{P}{-}$ preserves finite colimits. Since $P$ is finitely generated, there exist a positive integer $r$ and a $T$-module epimorphism $\pi:\Tt^{\{1,\ldots,r\}} \lra P$. For the sake of clarity and brevity, we denote by $T_i$ the copy of $\uT$ in position $i$ and by $1_i \in T_i$ its unit, for $i=1,\ldots,r$. By Proposition \ref{prop:epicoeq}, $\pi$ is a coequalizer and, by hypothesis, 
\[
\lhom{T}{P}{\pi}:\lhom{T}{P}{\Tt^{\{1,\ldots, r\}}}\lra\lhom{T}{P}{P}
\]
is a coequalizer of the corresponding morphisms, whence an epimorphism in particular. Choose a pre-image in $\lhom{T}{P}{\Tt^{\{1,\ldots, r\}}}$ of $\id_P$ and call it $\sigma$; it satisfies $\pi\circ \sigma = \id_P$.

Moreover, $\Tt^{\{1,\ldots ,r\}} = \boks{i=1}{r}{T_i}$ is the finite colimit with structure maps $$\eta_i:\uT \lra \Tt^{\{1,\ldots,r\}},\ \  z\lto z1_i\in T_i,$$ for all $i=1,\ldots,r$. By hypothesis again, the induced morphism 
\[
\boks{i=1}{r}{\lhom{T}{P}{\eta_i}} : \boks{i=1}{r}{\lhom{T}{P}{\uT}_i} \lra \lhom{T}{P}{\Tt^{\{1,\ldots ,r\}}}
\]
is an isomorphism, where $\lhom{T}{P}{\uT}_i$ denotes the copy of $\lhom{T}{P}{\uT}$ in position $i$. Therefore, there exist elements $\phi_{1},\ldots,\phi_{s} \in \lhom{T}{P}{\uT}$ (possibly $r\neq s$) 
such that
\[
\LL \phi_{1},\ldots,\phi_{s}\RR \in \boks{i=1}{r}{\lhom{T}{P}{\uT}_i}
\]
satisfies
\[
\boks{i=1}{r}{\lhom{T}{P}{\eta_i}}\left(\LL \phi_{1},\ldots,\phi_{s}\RR \right) = \sigma.
\]
Concretely, this amounts to say that, for every $p\in P$,
\[
\begin{aligned}
p & = \pi(\sigma(p)) = \pi\left(\LL \eta_{{i_1}}\circ\phi_{1},\ldots,\eta_{{i_s}}\circ\phi_{s}\RR (p)\right)\\
& = \LL \left(\pi\eta_{{i_1}}\phi_{1}\right)(p),\ldots,\left(\pi\eta_{{i_s}}\phi_{s}\right)(p)\RR  
   = \LL \phi_{1}(p) \cdot \pi(1_{{i_1}}),\ldots,\phi_{s}(p) \cdot \pi(1_{{i_s}})\RR ,
\end{aligned}
\]
where the $i_k$ are such that $\phi_{k} \in \lhom{T}{P}{\uT}_{i_k}$,  $k =1,\ldots,s$.
Set $e_k\coloneqq \pi(1_{i_{k}})\in P$ for $k=1,\ldots,s$. The foregoing relation says that, for every $p\in P$,
\[
p=\LL \phi_{1}(p) \cdot e_{1},\ldots,\phi_{s}(p) \cdot e_{s}\RR .
\]
We conclude that if $P$ is finitely generated and if $\lhom{T}{P}{-}$ preserves finite colimits, then $P$ satisfies the DBP.

\ref{item:fgp2} $\Rightarrow$ \ref{item:fgp3}. Consider the assignment 
\[
e: P\times {^*P} \lra \uT, \qquad (p,\alpha)\lto \alpha(p).
\]
Then
\[
e\left(\left[p,p',p''\right],\alpha\right) = \alpha\left(\left[p,p',p''\right]\right) = \left[\alpha\left(p\right),\alpha\left(p'\right),\alpha\left(p''\right)\right] = \left[e\left(p,\alpha\right),e\left(p',\alpha\right),e\left(p'',\alpha\right)\right]
\]
and
\[
e\left(p,\left[\alpha,\alpha',\alpha''\right]\right) = \left[\alpha,\alpha',\alpha''\right]\left(p\right) = \left[\alpha\left(p\right),\alpha'\left(p\right),\alpha''\left(p\right)\right] = \left[e\left(p,\alpha\right),e\left(p,\alpha'\right),e\left(p,\alpha''\right)\right],
\]
whence there exists a unique heap homomorphism $\ev: P\otimes {^*P}\lra \uT$ such that $\ev(p\otimes \alpha)=\alpha(p)$, for all $p\in P,\alpha\in {^*P}$. Moreover,
\[
\ev(t\cdot p\otimes \alpha)=\alpha(t\cdot p) = t\alpha(p)
\]
and
\[
\ev(p\otimes \alpha\cdot t)=(\alpha\cdot t)(p) = \alpha(p)t,
\]
for all $p\in P,\alpha\in {^*P},t\in T$, whence $\ev$ is $T$-bilinear. Consider also the assignment
\[
\db:\star \lra {^*P}\otimes_T P, \qquad *\lto \LL\phi_k \otimes_T e_{k}\RR_{k=1}^{s}.
\]
A direct check shows that
\[
\begin{gathered}
\left(\left(\ev\otimes_T P\right)\circ \left(P\otimes \db\right)\right)(p) = \left(\ev\otimes_T P\right)\left(\LL p \otimes \phi_k \otimes_T e_{k}\RR_{k=1}^{s}\right) = \LL\phi_k(p) \cdot e_{k}\RR_{k=1}^{s} = p, \\
\left(\left({^*P}\otimes_T \ev\right)\circ\left(\db\otimes {^*P}\right)\right)(\alpha) = \left({^*P}\otimes_T \ev\right)\left(\LL \phi_k\otimes_T e_{k} \otimes \alpha\RR_{k=1}^{s}\right) = \LL \phi_k  \cdot \alpha\left(e_{k}\right)\RR_{k=1}^{s} = \alpha,
\end{gathered}
\]
for all $p\in P$, $\alpha\in {^*P}$.

\ref{item:fgp3} $\Rightarrow$ \ref{item:fgp4}. For every $T$-module $M$, consider 
\[
\tilde{\tau}:{^*P}\times M \lra \lhom{T}{P}{M},\qquad (\alpha,m)\lto \left[p\lto \ev(p\otimes \alpha)\cdot m\right].
\]
For all $p\in P$, $\alpha,\alpha',\alpha''\in {^*P}$, $m,m',m''\in M$, $t\in T$, 
\[
\begin{aligned}
\tilde{\tau}\left(\left[\alpha,\alpha',\alpha''\right],m\right)(p) & = \left[\alpha,\alpha',\alpha''\right](p)\cdot m = \left[\alpha(p),\alpha'(p),\alpha''(p)\right]\cdot m \\
& = \left[\alpha(p)\cdot m,\alpha'(p)\cdot m,\alpha''(p)\cdot m\right] \\ 
& = \left[\tilde{\tau} (\alpha,m)(p),\tilde{\tau} (\alpha',m)(p),\tilde{\tau} (\alpha'',m)(p)\right] \\
& = \left[\tilde{\tau} (\alpha,m),\tilde{\tau} (\alpha',m),\tilde{\tau} (\alpha'',m)\right](p), \\
\tilde{\tau}\left(\alpha,\left[m,m',m''\right]\right)(p) & = \alpha(p)\cdot \left[m,m',m''\right] = \left[\alpha(p)\cdot m,\alpha(p)\cdot m',\alpha(p)\cdot m''\right] \\
 & = \left[\tilde{\tau} (\alpha,m)(p),\tilde{\tau} (\alpha,m')(p),\tilde{\tau} (\alpha,m'')(p)\right] \\
 & = \left[\tilde{\tau} (\alpha,m),\tilde{\tau} (\alpha,m'),\tilde{\tau} (\alpha,m'')\right](p), \\
\tilde{\tau}(\alpha\cdot t,m)(p) & = (\alpha\cdot t)(p)\cdot m = \alpha(p)t\cdot m = \tilde{\tau}(\alpha,t\cdot m)(p).
\end{aligned}
\]
Therefore, there exists a unique heap homomorphism $\tau_M: {^*P}\otimes_T M \lra \lhom{T}{P}{M}$ such that $\tau(\alpha\otimes_T m):p \lto \alpha(p)\cdot m$. The other way around,  write explicitly $\db(*) = \LL\phi_k \otimes_T e_{k}\RR_{k=1}^{s}$ and consider the assignment
\[
\sigma_M: \lhom{T}{P}{M} \lra {^*P}\otimes_T M, \qquad f\lto ({^*P}\otimes_T f)\left(\db(*)\right) = \LL\phi_k \otimes_T f\left(e_{k}\right)\RR_{k=1}^{s}.
\]
A direct computation shows that
\[
\sigma_M\tau_M(\alpha\otimes_T m) = \LL \phi_k \otimes_T \alpha\left(e_{k}\right)\cdot m\RR_{k=1}^{s} = \LL \phi_k \cdot \alpha\left(e_{k}\right)\RR_{k=1}^{s} \otimes_T m = \alpha\otimes_T m,
\]
for all $m\in M$, $\alpha \in {^*P}$, and 
\[
\tau_M\sigma_M(f)(p) = \LL\phi_k(p) \cdot f\left(e_{k}\right)\RR_{k=1}^{s} = f,
\]
whence they are inverses of each other. Furthermore, if $g:M\lra N$ is any $T$-linear map, then
\[
\tau_N\left({^*P}\otimes_T g\right)(\alpha\otimes_T m)(p) = \alpha(p)g(m) = (\lhom{T}{P}{g}\circ\tau_M)(\alpha\otimes_T m)(p),
\]
for all $p\in P$, $m \in M$, $\alpha\in {^*P}$, so that $\tau$ is also natural in $M$.

\ref{item:fgp4} $\Rightarrow$ \ref{item:fgp5}. Obvious, since tensoring by a right $T$-module is a left adjoint functor.

\ref{item:fgp5} $\Rightarrow$ \ref{item:fgp1}. Clearly, $\lhom{T}{P}{-}$ is a right exact functor. Thus, we are left to show that $P$ has to be finitely generated. Since $P$ is a set, we can consider the epimorphism $\pi:\Tt^P \lra P$ uniquely determined by the assignments $T_p \to P, z \lto z\cdot p,$ for all $p \in P$. Since epimorphisms are coequalizers, $\pi_*:\lhom{T}{P}{\Tt^P} \lra \lhom{T}{P}{P}$, $\psi \lto \pi\circ\psi,$ is still a coequalizer (whence an epimorphism), and since $\mathcal{T}^P$ is a small coproduct, 
\[
\lhom{T}{P}{\Tt^P} \cong \boks{p\in P}{}{\lhom{T}{P}{\uT}_p.}
\]
As in the proof of \ref{item:fgp1} $\Rightarrow$ \ref{item:fgp2}, one can consider a pre-image of $\id_P$ via $\pi_*$ and call it $\sigma$. There exist elements $\phi_{1},\ldots,\phi_{s} \in \lhom{T}{P}{\uT}$ such that
\[
\LL \phi_{1},\ldots,\phi_{s}\RR \in \boks{p\in P}{}{\lhom{T}{P}{\uT}_p}
\]
satisfies
\[
\boks{p\in P}{}{\lhom{T}{P}{\eta_p}}\left(\LL \phi_{1},\ldots,\phi_{s}\RR \right) = \sigma.
\]
Concretely, this amounts to say that, for every $q\in P$,
\[
\begin{aligned}
q & = \pi(\sigma(q)) = \pi\left(\LL \eta_{{p_1}}\circ\phi_{1},\ldots,\eta_{{p_s}}\circ\phi_{s}\RR (q)\right) \\
  & = \LL \left(\pi\eta_{{p_1}}\phi_{1}\right)(q),\ldots,\left(\pi\eta_{{p_s}}\phi_{s}\right)(q)\RR  
  = \LL \phi_{1}(q) \cdot \pi(1_{{p_1}}),\ldots,\phi_{s}(q) \cdot \pi(1_{{p_s}})\RR .
\end{aligned}
\]
Set $e_k\coloneqq \pi(1_{{p_k}})\in P$ for $k=1,\ldots,s$. Since the foregoing relation says that for every $p\in P$,
$
p=\LL \phi_{1}(p) \cdot e_{1},\ldots,\phi_{s}(p) \cdot e_{s}\RR,
$
we conclude that the $e_k$ form a finite family of generators of $P$.
\end{proof}

By taking inspiration from \cite[\S5.5]{Kelly} and in light of Theorem \ref{thm:toostrong}, we give the following definition (see also \cite[\S3]{Mitchell}).

\begin{definition}\label{def:tiny}
A $T$-module $P$ satisfying the equivalent conditions of Theorem \ref{thm:toostrong} is called \emph{tiny} (or \emph{small-projective}).
\end{definition}

\begin{remark}
\begin{enumerate}[label=(\alph*),ref=(\alph*),leftmargin=0.8cm]
\item\label{item:rem1} In the proof of the implication \ref{item:fgp1} $\Rightarrow$ \ref{item:fgp2} in Theorem~\ref{thm:toostrong}, there is no need for $s$ to be exactly $r$. 
\item\label{item:rem2} The dual basis map $\db$ does not depend on the choice of the dual basis. In fact, if $\left\{\left(e_{1},\ldots,e_s\right),\left(\phi_1,\ldots,\phi_s\right)\right\}$ and $\left\{\left(f_{1},\ldots,f_r\right),\left(\psi_1,\ldots,\psi_r\right)\right\}$ are two dual bases, then
\[
\begin{aligned}
\LL\phi_k \otimes_T e_{k}\RR_{k=1}^{s} & = \LL\phi_k \otimes_T \LL\psi_h\left(e_{k}\right) \cdot f_{h}\RR_{h=1}^{r}\RR_{k=1}^{s}  \stackrel{\eqref{eq:superbracket}}{=} \LL\LL\phi_k \cdot \psi_h\left(e_{k}\right) \otimes_T f_{h}\RR_{k=1}^{s}\RR_{h=1}^{r} \\
 & = \LL\LL\phi_k \cdot \psi_h\left(e_{k}\right)\RR_{k=1}^{s} \otimes_T f_{h}\RR_{h=1}^{r}  = \LL\psi_h \otimes_T f_{h}\RR_{h=1}^{r}.
\end{aligned}
\]

\item The implication from \ref{item:fgp5} to \ref{item:fgp4} in Theorem~\ref{thm:toostrong} follows also from the Eilenberg-Watts theorem, since $\lhom{T}{P}{-}$ is a heap functor.
\item In the implication from \ref{item:fgp4} to \ref{item:fgp3} in Theorem \ref{thm:toostrong}, the dual basis map $\db$ corresponds to the image of the identity morphism $\id_P$ via the isomorphism $\lhom{T}{P}{P}\cong {^*P}\otimes_T P$.
\item In the present section, we always worked with a left $T$-module $P$, implicitly viewed as a $(T,\star)$-bimodule. Observe that there is nothing particular in considering the distinguished truss $\star$ instead of any other truss. Therefore, the description and the properties of a small-projective $T$-module developed so far can be adapted, with no additional effort, to speak about a $(T,S)$-bimodule which is small-projective over $T$ on the left.
\end{enumerate}
\end{remark}

\begin{example}\label{ex:Rfgp}
If $P$ is a finitely generated and projective module over a ring $R$, then  $\tT(P)$ is a tiny $\tT(R)$-module.
\end{example}

\begin{example}[Free modules are not tiny] \label{ex:tinyfree}
Let $T$ be a unital truss and consider the free $T$-module $T\boxplus T$. Assume, by contradiction, that $T \boxplus T$ admits a dual basis $\left\{\left(e_{1},\ldots,e_s\right),\left(\phi_1,\ldots,\phi_s\right)\right\}$. Denote by $a=1_T$ the unit of the left-hand side copy of $T$ and by $b = 1_T$ the one of the right-hand side copy. Taking advantage of the heap isomorphism \eqref{iso.direct} (see \cite[Proposition 3.9]{BrzRyb:mod}), we may construct the heap homomorphism that ``measures tails''
\[
\ell : T \boxplus T \cong \hH\left(\gG(T,a) \oplus \gG(T,b) \oplus \ZZ \right) \lra \hH(\ZZ).
\]
Notice that, being composition of heap homomorphisms, $\ell$ is not influenced by the reduction of a symmetric word $w$ to one of the ``canonical forms'' $t$, $s$, $tsb$, $sta$, $tsab\cdots ba$, $stba\cdots ab$. Therefore, the ``length of tails'' is well-defined and, in particular, it is not influenced by the action of $T$ (see \cite[Formula (3.6)]{BrzRyb:mod} for a more rigorous formulation). Summing up, for all $z \in T \boxplus T$
\[
\ell\left(\LL \phi_{1}(z) \cdot e_{1},\ldots,\phi_{s}(z) \cdot e_{s}\RR\right) = \LL \ell\left(\phi_{1}(z) \cdot e_{1}\right),\ldots,\ell\left(\phi_{s}(z) \cdot e_{s}\right)\RR = \LL \ell\left(e_{1}\right),\ldots,\ell\left(e_{s}\right)\RR.
\]
However, if we set $m \coloneqq \LL \ell\left(e_{1}\right),\ldots,\ell\left(e_{s}\right)\RR$ (which does not depend on $z$) and we consider $z \coloneqq bab\cdots ab$ with $|m|+1$ instances of $b$, then $\ell(z) = |m|+1$, which is a contradiction.
\end{example}

\begin{example}
Let $T$ be a unital truss admitting a left absorber $a \in T$ and consider $P \coloneqq T \times T \times T$. Set 
\[
\begin{aligned}
e_1 \coloneqq (1_T,a,a), \qquad e_2 \coloneqq (a,1_T,a), \qquad & e_3=(a,a,1_T)\qquad \text{and} \\
\phi_1: T \times T \times T \lra T, \qquad & (x,y,z) \lto x, \\
\phi_2: T \times T \times T \lra T, \qquad & (x,y,z) \lto [a,y,a], \\
\phi_3: T \times T \times T \lra T, \qquad & (x,y,z) \lto z. 
\end{aligned}
\]
Then these form a dual basis for $P$ as a left $T$-module.

Assume furthermore that $a$ is a two-sided absorber. Denote by $S$ the set of all $3 \times 3$ matrices with coefficients in $T$. They inherits an abelian heap structure from the identification $S = T^9$ (that is, the bracket is taken component-wise). Moreover, $S$ admits a truss structure with the row-by-column multiplication
\[
\begin{pmatrix}
t_{1,1} & t_{1,2} & t_{1,3} \\
t_{2,1} & t_{2,2} & t_{2,3} \\
t_{3,1} & t_{3,2} & t_{3,3}
\end{pmatrix} \cdot \begin{pmatrix}
s_{1,1} & s_{1,2} & s_{1,3} \\
s_{2,1} & s_{2,2} & s_{2,3} \\
s_{3,1} & s_{3,2} & s_{3,3}
\end{pmatrix} = \big(r_{i,j}\big)
\quad \text{where} \quad r_{i,j} = \big[t_{i,1}s_{1,j},t_{i,2}s_{2,j},t_{i,3}s_{3,j}\big].
\]
As for matrices over rings, $P$ becomes a right $S$-module with row-by-column action
\[
\begin{pmatrix}
x & y & z
\end{pmatrix}\begin{pmatrix}
t_{1,1} & t_{1,2} & t_{1,3} \\
t_{2,1} & t_{2,2} & t_{2,3} \\
t_{3,1} & t_{3,2} & t_{3,3}
\end{pmatrix} = \begin{pmatrix}
\big[xt_{1,1},yt_{2,1},zt_{3,1}\big] & \big[xt_{1,2},yt_{2,2},zt_{3,2}\big] & \big[xt_{1,3},yt_{2,3},zt_{3,3}\big]
\end{pmatrix},
\]
which makes of it a $T$-$S$-bimodule and
\[
Q = \left\{\left.\begin{pmatrix} x \\ y \\ z \end{pmatrix} ~\right|~ x,y,z\in T\right\}
\]
becomes an $S$-$T$-bimodule analogously. Define the following morphisms
\[
\begin{gathered}
\ev: Q \otimes_T P \lra S, \quad \begin{pmatrix} x \\ y \\ z \end{pmatrix} \otimes_T \begin{pmatrix} x' & y' & z' \end{pmatrix} \lto \begin{pmatrix} x \\ y \\ z \end{pmatrix} \cdot \begin{pmatrix} x' & y' & z' \end{pmatrix} = \begin{pmatrix} xx' & xy' & xz' \\ yx' & yy' & yz' \\ zx' & zy' & zz' \end{pmatrix}, \\
\db: T \lra P \otimes_S Q, \quad 1_T \lto \begin{pmatrix} 1_T & a & a \end{pmatrix} \otimes_S \begin{pmatrix} 1_T \\ a \\ a \end{pmatrix}.
\end{gathered}
\]
They are invertible with inverses explicitly given by

\[
\begin{gathered}
\ev^{-1} \colon \big(t_{i,j}\big) \lto \left[\begin{pmatrix} t_{1,1} \\ t_{1,2} \\ t_{1,3} \end{pmatrix} \otimes_T \begin{pmatrix} 1_T & a & a \end{pmatrix} , \begin{pmatrix} t_{2,1} \\ t_{2,2} \\ t_{2,3} \end{pmatrix} \otimes_T \begin{pmatrix} a & [a,1_T,a] & a \end{pmatrix} , \begin{pmatrix} t_{3,1} \\ t_{3,2} \\ t_{3,3} \end{pmatrix} \otimes_T \begin{pmatrix} a & a & 1_T \end{pmatrix}\right], \\
\text{and} \qquad \db^{-1}: \begin{pmatrix} x & y & z \end{pmatrix} \otimes_S \begin{pmatrix} x' \\ y' \\ z' \end{pmatrix} \lto \begin{pmatrix} x & y & z \end{pmatrix} \cdot \begin{pmatrix} x' \\ y' \\ z' \end{pmatrix} = \big[xx',yy',zz'\big].
\end{gathered}
\]
Therefore, $\lmod{T}$ is equivalent to $\lmod{S}$ by Theorem \ref{thm:Morita}.
\end{example}

\section{Projective modules over a truss}

\subsection{Short exact sequences of modules over a truss}\label{sub.ses}

Let $T$ be a truss (not necessarily unital) and let $\star$ denote the singleton $T$-module. Recall from Section~\ref{sec.truss} that if $(M,\cdot)$ is a non-empty $T$-module and $e\in M$, then we denote by $M^{(e)}=(M,\cdot_e)$ the $T$-module with the induced action
\[
t\cdot_e m = [t\cdot m,t\cdot e,e].
\]
We say that a sequence of non-empty $T$-modules $\xymatrix{M\ar[r]^{f} & N\ar[r]^{g} & P}$  is {\em exact} provided there exists $e\in \im(g)$ such that $\im(f)=\ker_e(g)$ as sets. Notice that, in this case, $\im(f)\cong \ker_{e'}(g)$ as induced submodules for any other $e'\in\im(g)$.

\begin{lemma}\label{lemma:ses}
Let $M,N,P$ be $T$-modules and $f:M \lra N$ and  $g:N \lra P$ be $T$-linear maps. There exist exact sequences 
\begin{equation}\label{eq:3seqs}
\xymatrix{M\ar[r]^f & N \ar[r]^{g} & P}, \quad \xymatrix{\star \ar[r] & M^{(e)} \ar[r]^{f} & N^{(f(e))}} \quad \mbox{and}\quad  \xymatrix{N \ar[r]^{g} & P \ar[r] & \star}
\end{equation}
if and only if 
\begin{blist}
\item 
$f$ is injective and
\item  $N/\im(f)\cong P$ as $T$-modules,
\end{blist}
 where the module structure on $N/\im(f)$ is the one for which the canonical projection $\pi:N\lra N/\im(f)$ is $T$-linear.
\end{lemma}

\begin{proof}
Assume the sequences are exact. Then $f$ is injective, by the exactness of the second sequence, and 
$$
P = \im(g) \cong N/\Ker(g) \cong N/\im(f),
$$
where the equality is a restatement of the third sequence, the second isomorphism follows by the exactness of the first sequence, and  the first one is simply the first isomorphism theorem for $T$-modules.

Conversely, assume that there is an isomorphism $h:N/\im(f) \lra P$ of $T$-modules, and denote by $\pi:N\lra N/\im(f)$ the quotient map. In light of \cite[Proposition 4.32]{Brz:par} the sub-heap $\im(f)$ of $N$, as a kernel of $\pi$, admits an additional induced submodule structure. Denote it by $\im(f)^{(e)}\subseteq N^{(e)}$ for a certain $e\in \im(f)$. This entails that $\im(f)$ is a submodule of $N$ with respect to two (in principle, different) $T$-modules structures: $\im(f)\subseteq N$ with respect to the $T$-action for which $f$ is $T$-linear and $\im(f)^{(e)}\subseteq N^{(e)}$ with respect to the induced action coming from the identification $\im(f) = \ker_{\pi(e)}(\pi)$. Since $f$ is injective, we may transport the induced module structure on $M$. Denote it by $M^{(e')}$ for $e'\in M$ such that $f(e')=e$. Consider the sequences
\[
\begin{gathered}
\xymatrix{
\star \ar[r] & M^{(e')} \ar[r]^{f} & N^{(e)}},
\qquad
\xymatrix{M \ar[r]^-{f} & N \ar[r]^{h\circ \pi} & P},
\qquad
\xymatrix{N \ar[r]^-{h\circ\pi} & P \ar[r] & \star}.
\end{gathered}
\]
They are exact.
\end{proof}

By abuse of terminology, we will say that
\[
\xymatrix{
\star \ar@{.>}[r] & M \ar@<+0.3ex>@{.>}[r]^{f} \ar@<-0.3ex>[r] & N \ar[r]^g & P \ar[r] & \star
}
\]
is a short exact sequence of $T$-modules to mean that there exists $e\in M$ such that all three sequences \eqref{eq:3seqs}
are exact.

\begin{proposition}\label{prop:prod}
Let $\phi:M\lra N$ and $\psi:N\lra P$ be morphisms of $T$-modules. Assume that $\psi$ is surjective, that $\phi$ admits a retraction $\gamma$ (in particular, it is injective) and that
\begin{equation}\label{eq:exact1}
\xymatrix{
M \ar[r]^{\phi} & N \ar[r]^{\psi} & P
}
\end{equation}
is exact. Then $N\cong M\times P$ as $T$-modules. We will call such a sequence a \emph{split exact sequence}.
\end{proposition}

\begin{proof}
Since \eqref{eq:exact1} is exact, there exists $e\in P$ such that $\ker_e(\psi) = \im(\phi)$. Consider $e'\in N$ such that $\psi(e')=e$ and consider $\gamma(e')\in M$.  Since $e' \in \ker_e(\psi) = \im(\phi)$, $\phi(\gamma(e')) = e'$. Denote by $\gG(P;e)$, $\gG(M;\gamma(e'))$ and $\gG(N;e')$ the retracts of the heaps $P$,  $M$ and $N$ respectively. In light of \cite[Lemma 2.1]{Brz:par},  $\phi$ induces an additive map
\[
\widehat{\phi} : \gG(M;\gamma(e')) \lra \gG(N;e'), \qquad m \lto \left[\phi(m),\phi\gamma(e'),e'\right] = \left[\phi(m),e',e'\right] = \phi(m),
\]
and, analogously, $\widehat{\psi} = \psi$ and $\widehat{\gamma}=\gamma$,
which entail that
\[
\xymatrix{
0 \ar[r] & \gG(M;\gamma(e')) \ar[r]^-{\phi} & \gG(N;e') \ar[r]^-{\psi} \ar@/^3ex/[l]^-{\gamma} & \gG(P;e) \ar[r] & 0
}
\]
is a split short exact sequence of $\ZZ$-modules. Thus,
\[
\gG(N;e') \cong \gG(M;\gamma(e')) \oplus \gG(P;e) \cong \gG(M;\gamma(e')) \times \gG(P;e).
\]
From $\gG(N;e') \cong \gG(M;\gamma(e')) \times \gG(P;e)$ and \cite[page 8]{BrzRyb:mod}, it follows that
\[
\begin{aligned}
N = \hH(\gG(N;e'))
 &\cong \hH\left(\gG(M;\gamma(e')) \times \gG(P;e)\right)\\
& \cong \hH\left(\gG(M;\gamma(e'))\right) \times \hH\left(\gG(P;e)\right) = M\times P.
\end{aligned}
\]
Summing up, at the heap level  there is a (unique) isomorphism $N \cong M\times P$ induced by the universal property of the product and explicitly given by
\[
\Phi:N \lra M\times P, \qquad n\lto (\gamma(n),\psi(n)).
\]
By $T$-linearity of $\gamma$ and $\psi$, $\Phi$ is $T$-linear as well.
\end{proof}

For the sake of completeness, we point out that the inverse of $\Phi$ is explicitly given by
\[
\Phi^{-1} : M \times P \lra N, \qquad (m,p) \lto \left[n_p,\phi\gamma(n_p),\phi(m)\right]
\]
where $n_p \in N$ is any element such that $\psi(n_p)=p$.

\begin{corollary}\label{cor:prod}
Let $T$ be a truss and $n\in \NN$. Then for any $k\leq n$ there exists a $T$-module with absorber $M$ such that $T^{k}\times M\cong T^n$.
\end{corollary}
\begin{proof}
Observe that $\phi:T^k\to T^n$, given by $(t_1,\ldots,t_k)\lto (t_1,\ldots,t_k,t_k,\ldots,t_k)$, is a $T$-module homomorphism and clearly the sequence
\[
\xymatrix{
 T^k \ar@<0.5ex>[rr]^-{\phi} && T^n \ar[rr]^-{\pi_{\im(\phi)}} \ar@<0.5ex>[ll]^-{\pi_k} && T^n/\im(\phi),
}
\]
where $\pi_k$ is the projection on the first $k$ coordinates, is a split exact sequence. Therefore, by Proposition \ref{prop:prod}, $T^n\cong T^k\times (T^n/\im(\phi))$ and $T^n/\im(\phi)$ is the required module with an absorber.
\end{proof}

\begin{example}
Let $T=2\ZZ+1$ and let us consider $(2\ZZ+1)^3\cong (2\ZZ+1)\times M$ for some $(2\ZZ+1)$-module $M$ as in Corollary \ref{cor:prod}. In this case, $\phi$ is the map given by $2k+1\lto (2k+1,2k+1,2k+1)$ for all $k\in \ZZ$. It is easy to check that $\hH(M)\cong\hH((2\ZZ+1)\times (2\ZZ+1))$ and that the heap isomorphism is a $(2\ZZ+1)$-module homomorphism for the $(2\ZZ+1)$-action given on $(2\ZZ+1)\times (2\ZZ+1)$ by 
$$(2k+1)\cdot(2l+1,2h+1)=(2(2k+1)l+1,2(2k+1)h+1),$$
 for all $k,l,h\in \ZZ$. The desired absorber is $(1,1)$.
\end{example}

\begin{proposition}\label{prop:starsum}
Let $\phi:M\lra N$ and $\psi:N\lra P$ be morphisms of $T$-modules. Assume that $\phi$ is injective, that $\psi$ admits a section $\sigma$ (in particular, it is surjective) and that
\begin{equation}\label{eq:exact}
\xymatrix{
M \ar[r]^{\phi} & N \ar[r]^{\psi} & P
}
\end{equation}
is exact. Then there exists $e'\in M$ yielding an isomorphism of $T$-modules  $N \cong M^{(e')} \times P$, where $M^{(e')}$ denotes the $e'$-induced left $T$-module structure on $M$.
\end{proposition}

\begin{proof}
The argument for this proof follows closely that in the proof of Proposition~\ref{prop:prod}. Since \eqref{eq:exact} is exact, there exists $e\in P$ such that $\ker_e(\psi) = \im(\phi)$. Consider $\sigma(e)\in N$ and let $e'\in M$ be the unique element such that 
\begin{equation}\label{eq:sigmae}
\phi(e')=\sigma(e).
\end{equation}
In light of \cite[Lemma 2.1]{Brz:par},  $\sigma$ induces an additive map
\[
\widehat{\sigma} : \gG(P;e) \lra \gG(N;\sigma(e)), \qquad p \lto \left[\sigma(p),\sigma(e),\sigma(e)\right] = \sigma(p),
\]
 and analogously for $\phi$ and $\psi$. These yield the following split short exact sequence of $\ZZ$-modules
\[
\xymatrix{
0 \ar[r] & \gG(M;e') \ar[r]^-{\phi} & \gG(N;\sigma(e)) \ar[r]^-{\psi} & \gG(P;e) \ar[r] \ar@/^3ex/[l]^-{\sigma} & 0.
}
\]
Thus,
\[
\gG(N;\sigma(e)) \cong \gG(M;e') \oplus \gG(P;e) \cong \gG(M;e') \times \gG(P;e).
\]
From $\gG(N;\sigma(e)) \cong \gG(M;e') \times \gG(P;e)$ and \cite[page 8]{BrzRyb:mod}, it follows that
\[
N  =  \hH(\gG(N;\sigma(e)))  \cong  \hH\left(\gG(M;e')\right)  \times  \hH\left(\gG(P;e)\right)  =  M\times P.
\]
Explicitly, this isomorphism is given by the rule
\[
\Theta(m,p) = \left[\phi(m),\sigma(e),\sigma(p)\right].
\]
Now, consider $M$ as a $T$-module with the induced structure $t \cdot_{e'} m = [t \cdot m, t\cdot e',e']$. Then
\[
\begin{aligned}
\Theta(t \cdot_{e'} m, t \cdot p) & \stackrel{\phantom{(6.4)}}{=} \Theta([t \cdot m, t\cdot e',e'], t \cdot p) = \LL \phi(t \cdot m), \phi(t\cdot e'),\phi(e'),\sigma(e),\sigma(t \cdot p) \RR \\ 
 & \stackrel{\eqref{eq:sigmae}}{=} \left[t \cdot \phi(m), t\cdot \phi(e'),t \cdot \sigma(p)\right] = t \cdot \Theta(m,p)
\end{aligned}
\]
and hence it provides an isomorphism $N \cong M^{(e')} \times P$ as claimed.
\end{proof}

\begin{remark}
Let us compute explicitly the projection $N \lra M$ arising from Proposition \ref{prop:prod}. At the level of $\ZZ$-modules,
\[
\gG(N;\sigma(e)) \lra \gG(M;e'), \qquad n \lto \phi^{-1}\left( n - \sigma\psi(n) \right).
\]
By recalling that the module structure is the one induced by the heap structure, we conclude that
\[
n - \sigma\psi(n) = \left[n,\sigma(e),\sigma\psi(n)^{-1}\right] = \left[n,\sigma(e),\left[\sigma(e),\sigma\psi(n),\sigma(e)\right]\right] = \left[n,\sigma\psi(n),\sigma(e)\right].
\]
Therefore, the projection $N\lra M$ is given by $n\lto m_n$, where $m_n\in M$ is the unique element such that $\phi(m_n) = \left[n,\sigma\psi(n),\sigma(e)\right]$. Notice that this is not necessarily $T$-linear if $e$ or $\sigma(e)$ are not absorbers.
\end{remark}

At this point a curious reader may wonder why we introduced the terminology ``split exact sequence'' to refer to \eqref{eq:exact1} and we did not use a more specific one instead, in order to distinguish \eqref{eq:exact1} from \eqref{eq:exact} (such as e.g.\ ``left'' and ``right'' split exact sequences).
The reason is that if $\psi: N \lra P$ admits a section $\sigma:P \lra N$, then $\sigma$ itself admits $\psi$ as a retraction. By applying Proposition~\ref{prop:prod} to the split exact sequence
\[
\xymatrix{
P \ar[r]^-{\sigma} & N \ar[r]^-{\pi} \ar@/^3ex/[l]^-{\psi} & N/P,
}
\]
where $N/P$ is the quotient $T$-module with respect to the submodule $\sigma(P)\subseteq N$ and $\pi$ is the canonical projection, we conclude that $N \cong P \times N/P$ as $T$-modules. Now, a straightforward check shows that $\pi \circ \phi$ is an isomorphism of abelian heaps.

\subsection{Projectivity}

Let $T$ be a truss (not necessarily unital). 
Recall that epimorphisms in $\pMod{T}$ are surjective $T$-linear maps by  Proposition~\ref{prop:episurj}.

\begin{definition}
Let $P$ be a $T$-module. We say that $P$ is \emph{projective} if the functor $\lhom{T}{P}{-}:\lmod{T} \lra \ahrd$ preserves epimorphisms. That is to say, if for every surjective $T$-linear map $\pi:M \lra N$ and every $T$-linear map $f:P\lra N$ there exists a (not necessarily unique) $T$-linear map $\tilde{f}:P\lra M$ such that $\pi\circ\tilde{f} = f$. Diagrammatically,
\[
\xymatrix @R=20pt{
M \ar@{->>}[r]^-{\pi} & N \\
 & P. \ar@{.>}[ul]^-{\tilde{f}} \ar[u]_-{f}
}
\]

\end{definition}

\begin{proposition}\label{prop:tinyproj}
A $T$-module $P$ satisfying the DBP property is projective. In particular, every tiny $T$-module is finitely generated and projective.
\end{proposition}

\begin{proof}
In view of Proposition \ref{prop:epicoeq}, every epimorphism is a coequalizer. In particular, it is a colimit. By Theorem \ref{thm:toostrong}\ref{item:fgp5}, $\lhom{T}{P}{-}:\lmod{T} \lra \ahrd$ is cocontinuous, and so  it preserves small colimits and, in particular, epimorphisms. The last claim is a consequence of Theorem \ref{thm:toostrong}\ref{item:fgp1}.
\end{proof}

\begin{remark}
Proposition \ref{prop:tinyproj} should convince the reader that the terminology ``small-projective'' from Definition \ref{def:tiny} would also be very well-suited for tiny objects in $\pMod{T}$.
\end{remark}

\begin{lemma}\label{lem:what?!}
Every projective $T$-module $P$ admits a $T$-linear morphism $f:P \lra T$.
\end{lemma}

\begin{proof}
The required morphism is a filler of the following diagram \begin{equation*}
\begin{gathered}
\xymatrix @R=15pt{
T \ar[r] & \star \\
 & P. \ar[u] \ar@{.>}[ul]^-{f}
}
\end{gathered} \qedhere
\end{equation*}
\end{proof}

\begin{proposition}\label{prop:freeproj}
Every free $T$-module is projective.
\end{proposition}

\begin{proof}
Let $X$ be any set, $\pi:M \lra N$ be a surjective $T$-linear map and $f:\Tt^X \lra N$ a $T$-linear map. Consider also the inclusion  $\iota_X : X \lra \Tt^X, x\lto *x$. For every $x\in X$ set $n_x\coloneqq f(\iota_X(x))\in N$. Since $\pi$ is surjective, by the axiom of choice, for every $x\in X$, we may choose  an $m_x \in M$ such that $\pi(m_x) = n_x$. This defines a function $\bar{f}:X \lra M,\  x\lto m_x$. By the universal property of the free $T$-module, the latter extends uniquely to a $T$-linear map $\hat{f}:\Tt^X \lra M$ which satisfies  $\pi(\hat{f}(* x)) = \pi(m_x) = n_x = f(* x)$. Since this implies that $f\circ \iota_X$ and $\pi\circ\hat{f}\circ\iota_X$ coincide, the uniqueness ensured by the universal property of the free $T$-modules entails that $f= \pi\circ\hat{f}$ as desired. 
\end{proof}

\begin{corollary}\label{cor:noabs}
Let $T$ be a truss without absorbers. Then any $T$-module with absorber cannot be projective. In particular, free $T$-modules over a truss without absorbers cannot have absorbers.
\end{corollary}

\begin{proof}
Since $T$-linear maps preserve absorbers, a projective $T$-module $P$ cannot have absorbers in view of Lemma \ref{lem:what?!}. In particular, Proposition \ref{prop:freeproj} entails that free modules over a truss without absorbers cannot have absorbers.
\end{proof}

\begin{remark}
A truss $T$ has no absorbers if and only if there exists a non-empty $T$-module without absorbers. In fact, if $T$ admits an absorber $e$ then for every non-empty $T$-module $M$ and $m \in M$, $e\cdot m$ is an absorber in $M$. Conversely, $T$ itself is a $T$-module without absorbers. More precisely, a truss $T$ admits an absorber if and only if there exists a projective $T$-module admitting an absorber. Therefore, the hypothesis of Corollary \ref{cor:noabs} is not particularly restrictive.
\end{remark}

\begin{proposition}\label{prop:Absproj}
Let $R$ be a ring and $\tT(R)$ be the associated truss as in \S\ref{ssec.ringtype}. 
\begin{enumerate}[label=(\arabic*),ref=(\arabic*),leftmargin=0.8cm]
\item\label{Absproj:item1} If $P$ is projective over $\tT(R)$ then $P_{\Abs}$ is projective over $R$. 
\item\label{Absproj:item2} If $P$ is finitely generated over $\tT(R)$ then $P_{\Abs}$ is finitely generated over $R$.
\end{enumerate}
In particular, 
\begin{enumerate}[resume*]
\item\label{Absproj:item3} If $P$ is a tiny $\tT(R)$-module then $P_\Abs$ is a finitely generated and projective $R$-module.
\item\label{Absproj:item4} $P$ is a finitely generated and projective $R$-module if and only if $\tT(P)$ is a tiny $\tT(R)$-module.
\end{enumerate}
\end{proposition}

\begin{proof}
To prove \ref{Absproj:item1}, let $\pi:M \lra N$ be a surjective morphism of $R$-modules and assume that $f: P_{\Abs} \lra N$ is an $R$-linear map. Since the action of the functor $\tT$ on morphisms does not change the underlying mapping (see \S\ref{ssec.ringtype}) and since epimorphisms in $\pMod{\tT(R)}$ and $\pMod{R}$ are exactly surjective maps (see Proposition \ref{prop:episurj}), the functor $\tT$ preserves epimorphisms,  and hence we can consider the diagram of $\tT(R)$-modules
\[
\xymatrix @R=20pt{
\tT(M) \ar@{->>}[r]^-{\tT(\pi)} & \tT(N) \\
P \ar[r]_-{\eta_P} & \tT(P_{\Abs}), \ar[u]_-{\tT(f)}
}
\]
where $\eta_P : P \lra \tT(P_{\Abs})$ is the unit of the adjunction $(-)_{\Abs} \dashv \tT$. Since $P$ is projective over $\tT(R)$, there exists a filler $f'$ rendering the following diagram commutative:
\[
\xymatrix @R=20pt{
\tT(M) \ar@{->>}[r]^-{\tT(\pi)} & \tT(N) \\
P \ar@{.>}[u]^-{f'} \ar[r]_-{\eta_P} & \tT(P_{\Abs}). \ar[u]_-{\tT(f)}
}
\]
By applying the functor $(-)_{\Abs}$ to the latter diagram, we find the commutative diagram
\[
\xymatrix @R=20pt{
M \ar@{->>}[r]^{\pi} & N & \\
\tT(M)_{\Abs} \ar[u]^-{\eps_M}_{\cong} \ar@{->>}[r]^-{\tT(\pi)_{\Abs}} & \tT(N)_{\Abs} \ar[u]_-{\eps_N}^-{\cong} & P_{\Abs} \ar[ul]_-{f} \\
P_{\Abs} \ar@{.>}[u]^-{f'_{\Abs}} \ar[r]_-{(\eta_P)_{\Abs}} & \tT(P_{\Abs})_{\Abs} \ar[u]_-{\tT(f)_{\Abs}} \ar@/_2ex/[ur]_-{\eps_{P_{\Abs}}}
}
\]
and since $\eps_{P_{\Abs}} \circ (\eta_P)_{\Abs} = \id_P$, we constructed a morphism of $R$-modules $\hat{f} \coloneqq  \eps_M \circ f'_{\Abs} : P_{\Abs} \lra M$ such that $\pi \circ \hat{f} = f$.

To prove \ref{Absproj:item2}, pick an epimorphism $\pi : \boks{}{n}{\tT(R)} \lra P$. Since $(-)_\Abs$ is a left adjoint functor (see \S\ref{ssec.ringtype}) and every epimorphism in $\pMod{\tT(R)}$ is a coequalizer (see Proposition~\ref{prop:epicoeq}), $(-)_\Abs$ preserves epimorphisms and coproducts, and hence
\[
R^n \cong \bigoplus^n \tT(R)_\Abs \cong \left(\boks{}{n}{\tT(R)}\right) _\Abs \xrightarrow{\pi_\Abs} P_\Abs
\]
is an epimorphism of $R$-modules, showing that $P_\Abs$ is finitely generated.

Concerning the last claims, assume that $P$ is tiny over $\tT(R)$. Then it is finitely generated and projective by Proposition \ref{prop:tinyproj}, and hence $P_\Abs$ is finitely generated and projective over $R$, proving \ref{Absproj:item3}. Furthermore, in view of Example \ref{ex:Rfgp} we know that if $P$ is finitely generated and projective over $R$, then $\tT(P)$ is tiny over $\tT(R)$. Conversely, we have just seen that if $\tT(P)$ is tiny over $\tT(R)$, then $\tT(P)_\Abs \cong P$ is finitely generated and projective over $R$, thus showing \ref{Absproj:item4}.
\end{proof}

\begin{lemma}
The empty $T$-module is projective.
\end{lemma}

\begin{proof}
For every $T$-module $M$, there exists a unique morphism $\varnothing \lra M$ which is the empty morphism. Therefore, the following diagram is commutative and gives a lifting of the empty morphism along the epimorphism $\pi$:
\begin{equation*}
\begin{gathered}
\xymatrix @R=18pt{
M \ar@{->>}[r]^-{\pi} & N \\
 & \varnothing \ar@{.>}[ul]^-{\varnothing} \ar[u]_-{\varnothing}
}
\end{gathered}\qedhere
\end{equation*}
\end{proof}

\begin{proposition}\label{prop:projfact}
Let $M$ be a non-empty projective $T$-module. Then $M$ is a direct factor of a free $T$-module. More precisely, there exist a set $X$ and a $T$-module with absorber $P$ such that $M \times P \cong \Tt^X$ as $T$-modules.
\end{proposition}

\begin{proof}
Since every $T$-module is a quotient of a free one (as is shown at the end of \S\ref{ssec.functorial}), there exists a set $X$ and a surjective $T$-linear morphism $\gamma : \Tt^X \lra M$. By projectivity of $M$, $\gamma$ admits a $T$-linear section $\phi: M \lra \Tt^X$. Thus, $\phi$ is injective and we may identify $M$ with the $T$-submodule $\phi(M) \subseteq \Tt^X$. Consider now
\[
P \coloneqq {\Tt^X}\slash{\sim_{\phi(M)}} \cong {\Tt^X}\slash{M},
\]
which is a $T$-module with absorber. Denote by $\psi : \Tt^X \lra P$ the quotient map. As $M$ is non-empty, the sequence
\[
\xymatrix{
M \ar[r]^-{\phi} & \Tt^X \ar[r]^-{\psi} \ar@/^3ex/[l]^-{\gamma} & P
}
\]
is split exact with $\psi$ surjective and so, by Proposition \ref{prop:prod}, $\Tt^X \cong M \times P$ as $T$-modules.
\end{proof}

The converse of Proposition \ref{prop:projfact} holds as well.

\begin{proposition}\label{prop:factproj}
Let $M$ be a $T$-module. If there exists a $T$-module $P$ with absorber and a set $X$ such that $\Tt^X \cong M \times P$, then $M$ is projective.
\end{proposition}

\begin{proof}
Let $e \in P$ be an absorber. Then the assignment $\phi:M \lra M \times P$, $m \lto (m,e),$ is a well-defined injective $T$-linear morphism, providing a section for the canonical projection $\gamma:M\times P \lra M$, $(m,p)\lto m$. As a consequence, for every surjective  morphism $g:N \lra Q$ of $T$-modules  and every $T$-linear map $f: M\lra Q$, we can consider the diagram of $T$-linear maps
\[
\begin{gathered}
\xymatrix @R=20pt{
N \ar[r]^-{g} & Q \\
M \times P \ar@<+0.4ex>[r]^-{\gamma} & M. \ar[u]_-{f} \ar@<+0.4ex>[l]^-{\phi}
}
\end{gathered}
\]
By projectivity of $\Tt^X$, there exists $\tilde{f} : \Tt^X \lra N$ such that the diagram
\[
\begin{gathered}
\xymatrix @R=20pt{
N \ar[rr]^-{g} & & Q \\
\Tt^X \ar[r]^-{\cong}_-{\tau} \ar[u]^-{\tilde{f}} & M \times P \ar@<+0.4ex>[r]^-{\gamma} & M \ar[u]_-{f} \ar@<+0.4ex>[l]^-{\phi}
}
\end{gathered}
\]
commutes, \ie $g \circ \tilde{f} = f \circ \gamma \circ \tau$. If we set $\hat{f} \coloneqq \tilde{f} \circ \tau^{-1} \circ \phi$, then
\[
g \circ \hat{f} = g \circ \tilde{f} \circ \tau^{-1} \circ \phi = f \circ \gamma \circ \phi = f,
\]
whence $f: M \lra Q$ can be lifted to a $T$-linear map $\hat{f}: M \lra N$ along $g$, \ie $g \circ \hat{f} = f$, proving that $M$ is projective.
\end{proof}

\begin{theorem}\label{thm.proj}
A $T$-module $M$ is projective if and only if there exists a $T$-module with absorber $P$ such that $M \times P$ is a free $T$-module.
\end{theorem}

\begin{proof}
It follows from Propositions \ref{prop:projfact} and \ref{prop:factproj}.
\end{proof}

\begin{proposition}\label{prop:facttiny}
Let $P$ be a tiny $T$-module with dual basis $\{(e_1,\ldots,e_s),(\phi_1,\ldots,\phi_s)\}$. Then there exists a $T$-module $Q$ with an absorber, such that $P \times Q \cong T^s$.
\end{proposition}

\begin{proof}
By the universal property of the direct product, there exists a unique morphism of $T$-modules $\phi:P \lra T^s$ such that $\pi_k \circ \phi = \phi_k$, where $\pi_k:T^s \lra T$ is the projection on the $k$-th factor. The other way around, consider the assignment
\[
\pi:T^s \lra P, \qquad (t_1,\ldots,t_s) \lto \left[t_1\cdot e_1, \ldots, t_s\cdot e_s\right].
\]
Since
\[
\begin{aligned}
\pi\Big(\big[(t_1,\ldots,t_s),&(t'_1,\ldots,t'_s),(t''_1,\ldots,t''_s)\big]\Big) = \pi\Big(\big([t_1,t'_1,t''_1],\ldots,[t_s,t'_s,t''_s]\big)\Big) \\
&= \big[[t_1,t'_1,t''_1]\cdot e_1, \ldots, [t_s,t'_s,t''_s]\cdot e_s\big]\\
& = \big[[t_1\cdot e_1,t'_1\cdot e_1,t''_1\cdot e_1], \ldots, [t_s\cdot e_s,t'_s\cdot e_s,t''_s\cdot e_s]\big] \\
&\stackrel{\eqref{eq:superbracket}}{=}\big[[t_1\cdot e_1,\ldots,t_s\cdot e_s], [t'_1\cdot e_1,\ldots,t'_s\cdot e_s], [t''_1\cdot e_1,\ldots,t''_s\cdot e_s]\big] \\
&= \big[\pi(t_1,\ldots,t_s), \pi(t'_1,\ldots,t'_s), \pi(t''_1,\ldots,t''_s)\big],
\end{aligned}
\]
for all $(t_1,\ldots,t_s),(t'_1,\ldots,t'_s),(t''_1,\ldots,t''_s)\in T^s$ and
\[
\begin{aligned}
\pi\left(t\cdot(t_1,\ldots,t_s)\right) &= \pi\left((tt_1,\ldots,tt_s)\right) = \left[tt_1\cdot e_1, \ldots, tt_s\cdot e_s\right] \\
&= t\cdot \left[t_1\cdot e_1, \ldots, t_s\cdot e_s\right] = t\cdot \pi\left((t_1,\ldots,t_s)\right),
\end{aligned}
\]
for all $t\in T$, $\pi$ is a morphism of left $T$-modules satisfying $\pi \circ \phi = \id_P$ (because $P$ satisfies the DBP). Thus, $\phi$ is injective and we may identify $P$ with the submodule $\phi(P)\subseteq T^s$. As in the proof of Proposition \ref{prop:projfact}, consider $Q\coloneqq T^s/\sim_{\phi(P)} \cong T^s/P$, which is a $T$-module with absorber, and the quotient map $\psi:T^s \lra Q$. By \ref{item:dbp1} of Example \ref{ex:dbp1}, $P$ is non-empty, and so the sequence
\[
\xymatrix{
P \ar[r]^-{\phi} & T^s \ar[r]^-{\psi} \ar@/^3ex/[l]^-{\pi} & Q
}
\]
is split exact with $\psi$ surjective. By Proposition \ref{prop:prod}, $T^s \cong P \times Q$ as $T$-modules.
\end{proof}

Differently from what we have seen for projective modules, it seems that the converse of Proposition \ref{prop:facttiny} requires stronger hypotheses.

\begin{proposition}
Let $T$ be a unital truss with left absorber $a \in T$. If there exist an odd positive integer $s=2k+1$ and $T$-modules $P,Q$ such that $T^s \cong P \times Q$ as left $T$-modules, then both $P$ and $Q$ are tiny $T$-modules.
\end{proposition}

\begin{proof}
If $T$ admits a left absorber $a$ and $\gamma:T^s \lra P \times Q$ is an isomorphism as left $T$-modules, then $(a,a,\ldots,a) \in T^s$ is a left absorber and so $a_P\coloneqq\pi_P\left(\gamma((a,a,\ldots,a))\right) \in P$ and $a_Q\coloneqq\pi_Q\left(\gamma((a,a,\ldots,a))\right) \in Q$ are absorbers as well.

Now, set  $1_i \coloneqq \left(a,\ldots,a,1_T,a,\ldots,a\right)$ where $1_T$ appears in the $i$-th position. 
Consider the projection $\pi \coloneqq \left(T^s \stackrel{\gamma}{\lra} P \times Q \stackrel{\pi_P}{\lra} P\right)$, the elements $e_k \coloneqq \pi(1_k) \in P$ and the compositions $\phi_k\coloneqq \left(P \lra P \times Q \stackrel{\gamma^{-1}}{\lra} T^s \stackrel{\pi_k}{\lra} T\right), p \lto \pi_k\gamma(p,a_Q),$ for all $k=1,\ldots,s$. For all $p \in P$, one finds that
\[
\begin{gathered}
\left[\phi_1(p)\cdot e_1, \ldots, \phi_s(p)\cdot e_s\right] = \pi\big(\left[\phi_1(p)\cdot 1_1, \ldots, \phi_s(p)\cdot 1_s\right]\big) = \pi\big(\left(\phi_1(p), \ldots, \phi_s(p)\right)\big) = p,
\end{gathered}
\]
and so $P$ satisfies the DBP. The proof for $Q$ is analogous.
\end{proof}

\section*{Acknowledgements}
The research of Tomasz Brzezi\'nski  is partially supported by the National Science Centre, Poland, grant no. 2019/35/B/ST1/01115.

Paolo Saracco is a Charg\'e de Recherches of the Fonds de la Recherche Scientifique - FNRS and a member of the ``National Group for Algebraic and Geometric Structures and their Applications'' (GNSAGA-INdAM). PS would like to thank the members of the Department of Mathematics of the Swansea University for their warm hospitality during his stay in fall 2019.


\begin{thebibliography}{99}
\bibitem{Bae:ein} 
R.\ Baer, Zur Einf\"uhrung des Scharbegriffs, {\em J.\ Reine Angew.\ Math.} \textbf{160} (1929), 199--207.

\bibitem{Ber:inv} 
G.~M.~Bergman, \emph{An invitation to general algebra and universal constructions}. Second edition. Universitext. Springer, Cham, 2015.
 

\bibitem{Borceaux} 
F.~Borceux, \emph{Handbook of categorical algebra. 2. Categories and structures}. Encyclopedia of Mathematics and its Applications, \textbf{51}. Cambridge University Press, Cambridge, 1994.

\bibitem{BrBr}
S.~Breaz, T.\ Brzezi\'nski, \emph{The Baer-Kaplansky theorem for all abelian groups and modules}, Bull.\ Math.\  Sci.\  (2021). Online ready, \href{https://doi.org/10.1142/S1664360721500053}{doi: 10.1142/S1664360721500053}

\bibitem{Brz:tru} 
T.\ Brzezi\'nski, \emph{Trusses: Between braces and rings}, Trans.\ Amer.\ Math.\  Soc.\ {\bf 372} (2019), 4149--4176.

\bibitem{Brz:par} 
T.\ Brzezi\'nski, \emph{Trusses: Paragons, ideals and modules}, J. Pure Appl. Algebra {\bf 224}  (2020), 106258.

\bibitem{BrzMerRyb:pre}
T.\ Brzezi\'nski, S.\ Mereta \& B.\ Rybo\l owicz, \emph{From pre-trusses to skew braces}, {\em Preprint} arXiv: 2007.05761 (2020), Publ. Mat.\ {\em in press}.

\bibitem{BrzRyb:cong} 
T.\ Brzezi\'nski \& B.\ Rybo\l owicz, \emph{Congruence classes and extensions of rings with an application to braces}, Commun.\ Contemp.\ Math.\ {\bf 23} (2021), Paper No.\ 2050010, 24 pp.

\bibitem{BrzRyb:mod} 
T.\ Brzezi\'nski \& B.\ Rybo\l owicz, \emph{Modules over trusses vs modules over rings: direct sums and free modules}, Algebras Repr.\ Theory {\bf 25} (2022), 1--23.

\bibitem{ColAnt}
I.\ Colazzo, A.\ Van Antwerpen, \emph{The algebraic structure of left semi-trusses}. J.\ Pure Appl.\ Algebra \textbf{225} (2021), no. 2, 106467.

\bibitem{Fu15} L.\ Fuchs, {\em Abelian Groups},
Springer, Cham, 2015. 

\bibitem{Kelly}
G.~M.~Kelly, \emph{Basic concepts of enriched category theory}. Reprint of the 1982 original. Repr. Theory Appl. Categ. No. \textbf{10} (2005).

\bibitem{MacLane} 
S.\ Mac Lane, {\em Categories for the Working Mathematician} (Second ed.), Springer, New York, 1998.

\bibitem{Mitchell}
B.~Mitchell, \emph{Rings with several objects}. Adv.\ Math.\ \textbf{8} (1972), 1--161.


\bibitem{Pru:the} 
H.\ Pr\"ufer, \emph{Theorie der Abelschen Gruppen. I.\ Grundeigenschaften}, Math.\ Z. 20:165--187, 1924.

\bibitem{Rum:bra} 
W.\ Rump,  \emph{Braces, radical rings, and the quantum Yang-Baxter equation}, J.\ Algebra {\bf 307} (2007), 153--170.


\bibitem{Street}
R.~Street, \emph{Elementary cosmoi. I}. Category Seminar (Proc. Sem., Sydney, 1972/1973), pp. 134--180. Lecture Notes in Math., Vol. \textbf{420}, Springer, Berlin, 1974.

\bibitem{Sus:the} 
A. K.\ Su\v skevi\v c, {\em Theory of Generalized Groups}, Goc.\ Nau\v cno-Techn.\ Izdat.\ Ukrainy, Kharkov, 1937.

\end{thebibliography}
\end{document}